\documentclass[a4paper,11pt]{scrartcl}
\usepackage[latin1]{inputenc} 
\usepackage[english]{babel}
\usepackage{amsmath,amssymb,amsthm}
\usepackage[pdftex]{graphicx}
\usepackage[refpage]{nomencl} 
\usepackage{aliascnt}
\usepackage[colorlinks=true,linkcolor=blue,citecolor=blue]{hyperref} 
\usepackage{easyeqn} 
\usepackage{enumerate}

\usepackage{makeidx}

\makeindex
\allowdisplaybreaks[2]

\newtheorem{thm}{Theorem}[section]
\newaliascnt{prop}{thm}
\newaliascnt{lem}{thm}
\newaliascnt{cor}{thm}
\newaliascnt{defn}{thm}
\newaliascnt{rem}{thm}
\newaliascnt{exa}{thm}

\newtheorem{lem}[lem]{Lemma} 
\newtheorem{cor}[cor]{Corollary}
\theoremstyle{definition}
\newtheorem{defn}[defn]{Definition}
\theoremstyle{remark}
\newtheorem{rem}[rem]{Remark}

\aliascntresetthe{prop}
\aliascntresetthe{lem}
\aliascntresetthe{cor}
\aliascntresetthe{defn}
\aliascntresetthe{rem}
\aliascntresetthe{exa}

\theoremstyle{plain}
\newtheorem{prob}{Approximation-Problem}

\addto{\extrasenglish}{}

\newenvironment{alg}{\begin{small}\tt}{\end{small}}

\numberwithin{equation}{section} 

\newcommand{\bpm}{\begin{pmatrix}}
\newcommand{\epm}{\end{pmatrix}}
\newcommand{\bsm}{\left(\begin{smallmatrix}}
\newcommand{\esm}{\end{smallmatrix} \right)}
\newcommand{\mb}{\mathbb}
\newcommand{\mc}{\mathcal}
\newcommand{\wt}{\widetilde}
\newcommand{\matpdf}[1]{\includegraphics[trim=31mm 90mm 39mm 85mm,scale=0.45]{#1}}

\newcommand{\matlab}{$\text{\textsc{Matlab}}^\text{\textregistered}$}

\newcommand{\indx}[1]{\emph{#1}\index{#1}}
\newcommand{\subindx}[2]{\emph{#1}\index{#1}}

\renewcommand{\Re}{\operatorname{Re}}

\renewcommand{\ker}{\operatorname{ker}}
\newcommand{\penc}[1]{\left(#1\right)}
\renewcommand{\det}{\operatorname{det}}
\newcommand{\rank}{\operatorname{rank}}

\newcommand{\trace}{\operatorname{trace}}

\newcommand{\dom}{\mc{D}}

\newcommand{\infb}[2]{\inf\limits_{#2}#1}

\newcommand{\ie}{\text{\normalfont{i}}}

\newcommand{\apenc}[1]{\mb{P}_{#1}}
\newcommand{\rpenc}[1]{\mc{R}\apenc{#1}}
\newcommand{\rat}[2]{\mb{T}_{#1,#2}}
\newcommand{\rez}[1]{\mb{S}(#1)}
\newcommand{\sys}[1]{\mb{S}_{#1}}

\newcommand{\Cont}[1]{\mathfrak{C}_{\scriptscriptstyle #1}}
\newcommand{\Obs}[1]{\mathfrak{O}_{\scriptscriptstyle #1}}
\newcommand{\Fkt}[1]{\mc{G}(#1)}
\newcommand{\rse}[3]{#1\cdot #2\cdot #3}
\newcommand{\rser}[2]{#1\cdot #2}
\newcommand{\opt}[3]{{\mc{#1}}_{\scriptscriptstyle #2,#3}}

\begin{document}



\title{Optimal $\boldsymbol{RH_2}$- and $\boldsymbol{RH_\infty}$-Approximation of Unstable Descriptor Systems}
\author{Marcus K\"ohler}

\maketitle

\begin{abstract}
Stability perserving is an important topic in approximation of systems, e.g.\ model reduction. If the original system is stable, we often want the approximation to be stable. But even if an algorithm preserves stability the resulting system could be unstable in practice because of round-off errors. Our approach is approximating this unstable reduced system by a stable system. More precisely, we consider the following problem. Given an unstable linear time-invariant continuous-time descriptor system with transfer function $G$, find a stable one whose transfer function is the best approximation of $G$ in the spaces $RH_2$ and $RH_\infty$, respectively. 
 Explicit optimal solutions are presented under consideration of numerical issues.
\end{abstract}

\section{Introduction}

We deal with linear time-invariant descriptor systems described by differential algebraic equations (see e.g.\  \cite{Kunkel2006}, \cite{Dai1989} for more details)
 \begin{EQ}[rl]\label{eq:dsys}
  E\dot{x}(t)&=Ax(t)+Bu(t)\\
  y(t)&=Cx(t)+Du(t)
 \end{EQ}
with matrices $(E,A,B,C,D)\in \mb{R}^{n\times n}\times\mb{R}^{n\times n}\times\mb{R}^{n\times m}\times\mb{R}^{p\times n}\times\mb{R}^{p\times m}$ and the corresponding transfer function $s\mapsto C(sE-A)^{-1}B+D$. In many applications such as model reduction (see \cite{Antoulas2001} for an overview an further references) or system identification (see \cite{Ljung1999}) one wishes to approximate systems such that the transfer functions of the original and the approximated system are as close as possible. An important property of an approximation technique is preservation of stability (cf.\ e.g.\  balanced truncation \cite{Antoulas2001}, hankel norm approximation \cite{Glover1984}). However, not every method has this property in general (Arnoldi method, Lanczos method \cite{Antoulas2001}). Moreover, a stable approximation can be unstable in practice \cite{Bai1998a}. This happens in computer implementations because of round-off errors. One way out is restarting the algorithm with changed parameters, e.g.\ interpolation points in Lanczos method. The disadvantage is at least doubled computational complexity. Another approach is to modify the algorithm directly (e.g.\ partial Padé-via-Lanczos method \cite{Bai2001}). 

In this paper our approach for stability perserving is to consider the computed unstable approximation of the original system. Given is an unstable descriptor system $S$ \eqref{eq:dsys} with transfer function $G$. Find a stable descriptor system whose transfer function is the best approximation of $G$ in the spaces $RL_2$ and $RL_\infty$ of real rational functions on the imaginary axis $\ie\mb{R}$ in the Lebesgue spaces $L_2$ and $L_\infty$, respectively. We call this system an optimal $RH_2$-approximation and optimal $RH_\infty$-approximation of $S$, respectively.



For special cases this problem was already solved. The characterization of the unique optimal $RH_2$-approximation of a standard system (i.e.\ $E=I$ in \eqref{eq:dsys}) with $D=0$ is a simple consequence of the Paley-Wiener theorem. The set of all suboptimal $RH_\infty$-approximations of standard systems was determined in \cite{Green1990} with j\nobreakdash-spectral factorizations to solve model matching problems. However, no optimal ones were explicitly given. 
 An explicit representation of a nonunique optimal $RH_\infty$-approximation of minimal antistable standard systems was presented in the proof of \cite[Theorem 6.1]{Glover1984}. This theorem was only used as an auxiliary result for a more general problem, the optimal hankel norm approximation. That representation requires the possibly ill-conditioned balanced minimal realization of $G$. This numerical issue was discussed in \cite{Safonov1990}. 

Our approach for approximating an unstable descriptor system $S$ in a numerically reliable way is the following. First, we decompose $S$ into a stable and antistable part by combining \cite[Theorem 7.7.2]{Golub1996} and \cite[Section 4.1]{Kaagstroem1992}. We show that it is sufficient to replace the antistable part by its approximation to get an approximation of $S$. This was e.g.\ proposed in \cite[Section 8.3]{Francis1987} for standard systems and w.r.t.\ $RL_\infty$. The stable part is the unique optimal $RH_2$-approximation of $S$. Our representation of the optimal $RH_\infty$-approximation only requires at most one singular value decomposition.  
For discret-time standard systems \cite{Mari2000} presented an approximation method similar to ours. 
However, in contrast to our approach they used a balanced minimal realization of the antistable part.

The paper is organized as follows. In \autoref{sec:pre} we summarize some results regarding the theory of matrices and matrix pairs. We introduce the setting and recall some system transformations, e.g.\  balanced realization. 
\autoref{sec:appr} states and solves both approximation problems. \autoref{sec:num} presents an algorithm solving our problem and some numerical examples.

\section{Preliminaries} \label{sec:pre}

We define $\mb{C}_{>0}:=\{s\in\mb{C};~\Re(s)>0\}$ and analogously $\mb{C}_{\geq0}$, $\mb{C}_{<0}$, $\mb{C}_{\leq0}$. The set $\ie\mb{R}$ is the imaginary axis. For a function $f$ we denote its domain by $\dom(f)$ and for $M\subseteq\dom(f)$ the image of $M$ under $f$ by $f[M]$. The kernel of a matrix $A\in\mb{C}^{n\times m}$ is denoted by $\ker(A)$, the adjoint by $A^*$, the spectrum by $\sigma(A)$ and the Frobenius norm by $\|A\|_F=\sqrt{\trace(A^*A)}$. The matrix $I_n\in\mb{R}^{n\times n}$ is the identity matrix and $0_{n\times m}\in\mb{R}^{n\times m}$ the zero matrix. The vector $e_i\in\mb{R}^n$ is the $i$th unit vector.

\begin{defn}
We call $\penc{E,A}\in\mb{R}^{n\times n}\times\mb{R}^{n\times n}$ a \indx{matrix pair}. A scalar $\lambda \in \mb{C}$ is called a \subindx{(generalized) eigenvalue}{matrix pair} of $\penc{E,A}$, if $\det(\lambda E-A)=0$. If there exists $v\neq0$ such that $Ev=0$ and $Av\neq0$, then $\infty$ is called a (generalized) eigenvalue of $\penc{E,A}$. We denote the set of all eigenvalues of $\penc{E,A}$ by $\sigma(E,A)\subseteq\mb{C}\cup\{\infty\}$ and the resolvent set by $\rho(E,A):=\mb{C}\backslash\sigma(E,A)$.\nomenclature[s]{$\sigma(E,A)$}{} If $\rho(E,A)=\emptyset$, then $\penc{E,A}$ is called \subindx{singular}{matrix pair} otherwise \subindx{regular}{matrix pair}. We denote the set of all regular matrix pairs by $\apenc{n}:=\left\{\penc{E,A}\in\mb{R}^{n\times n}\times\mb{R}^{n\times n};~\penc{E,A}\text{ regular}\right\}$. \nomenclature[s]{$\rpenc{n}$}{}
\end{defn}
\begin{rem}
If the matrix pair is regular, then the set of eigenvalues is finite. Otherwise every $\lambda\in\mb{C}$ is an eigenvalue. Details about the generalized eigenvalue problem can be found in \cite{Stewart1972}.
\end{rem}
To prepare decompositions of systems, we combine the results in \cite[Theorem 7.7.2]{Golub1996} and \cite[Section 4.1]{Kaagstroem1992} to get the following theorem. 
\begin{thm}\label{thm:decsyl}
Let $\penc{E,A}\in\apenc{n}$ and $M_1,M_2\subseteq \mb{C}\cup\{\infty\}$ be two disjoint sets such that $\sigma(E,A)=M_1\cup M_2$. Then there exist orthogonal $U,V\in\mb{R}^{n\times n}$ such that
\begin{EQ}\label{eq:gdc}
\penc{UEV,UAV}=\penc{\bpm E_1&E_2\\0&E_3\epm,\bpm A_1&A_2\\0&A_3\epm},
\end{EQ}
$\sigma(E_1,A_1)=M_1$ and $\sigma(E_3,A_3)=M_2$. Moreover, there exist $R,L\in\mb{R}^{m\times p}$ such that
\begin{EQ}[rl]\label{eq:sylv}
A_1R-LA_3&=-A_2\\
E_1R-LE_3&=-E_2\,.
\end{EQ}
The matrices $P:=\bsm I&-L\\0&I\esm U,~ Q:=V\bsm I&R\\0&I\esm$ satisfy
\begin{EQ}
\penc{PEQ,PAQ}=\penc{\bpm E_1&0\\0&E_3\epm,\bpm A_1&0\\0&A_3\epm}.
\end{EQ}
\end{thm}
\begin{rem}
We get \eqref{eq:gdc} e.g.\  by applying the generalized Schur decomposition to $\penc{E,A}$. Equation \eqref{eq:sylv} is called generalized (or coupled) Sylvester equation (see \cite{Kaagstroem1992}).
\end{rem}
We recall the following known result which can be used to test regularity of matrices.
\begin{thm}[{\cite[Section 2.3]{Zhou1996}}]\label{thm:schur}
Let $A\in\mb{C}^{n\times n}$, $B\in\mb{C}^{n\times m}$, $C\in\mb{C}^{m\times n}$, $D\in\mb{C}^{m\times m}$ and $M:=\bsm A&B\\C&D\esm$. If $A$ is regular, then we have $\det(M)=\det(D-CA^{-1}B)\det(A)$. If $D$ is regular, then we have $\det(M)=\det(A-BD^{-1}C)\det(D)$.  The matrices $D-CA^{-1}B$ and $A-BD^{-1}C$ are called the \indx{Schur complement} of $A$ and $D$ in $M$, respectively.
\end{thm}


\begin{defn}
For $n,p,m\in\mb{N}$ we define the sets of \indx{systems} (compare \eqref{eq:dsys})
\begin{EQ}[rl] 
\sys{n,p,m}&:=\apenc{n}\times\mb{R}^{n\times m}\times\mb{R}^{p\times n}\times\mb{R}^{p\times m},\\
\sys{n,p,m}^0&:=\left\{(E,A,B,C,D)\in\sys{n,p,m};~\ie\mb{R}\subseteq\rho(E,A)\right\},\\ 
\sys{n,p,m}^+&:=\left\{(E,A,B,C,D)\in\sys{n,p,m};~\mb{C}_{\geq0}\subseteq\rho(E,A)\right\},\\
\sys{n,p,m}^-&:=\left\{(E,A,B,C,D)\in\sys{n,p,m};~\mb{C}_{\leq0}\subseteq\rho(E,A),~E\text{ regular}\right\},
\end{EQ}
where $\sys{n,p,m}^+$ and $\sys{n,p,m}^-$ are the set of the \indx{stable} and \indx{antistable} systems, respectively. We call a system $S=(E,A,B,C,D)\in\sys{n,p,m}$ a \indx{standard system} if $E=I$ and a \indx{descriptor system} otherwise. For $S_i=(E_i,A_i,B_i,C_i,D_i)\in\sys{n_i,p,m}$, $i\in\{1,2\}$, and regular $P,Q\in\mb{R}^{n_1\times n_1}$ we define the operations
\begin{EQ}[rl]
 S_1\oplus S_2&:=\left(\bpm E_1&0\\0&E_2\epm,\bpm A_1&0\\0&A_2\epm,\bpm B_1\\B_2\epm,\bpm C_1&C_2\epm,D_1+D_2\right),\\
 \rse{P}{S_1}{Q}&:=(PE_1Q,PA_1Q,PB_1,C_1Q,D_1)\,,\\
 \rser{P}{S_1}&:=\rse{P}{S_1}{I_{n_1}}\,, \qquad\qquad\qquad\qquad \rser{S_1}{Q}:=\rse{I_{n_1}}{S_1}{Q}\,.
\end{EQ}
\nomenclature[s]{$\sys{n,p,m}$}{}\nomenclature[s]{$\sys{n,p,m}^+$}{}\nomenclature[s]{$\sys{n,p,m}^-$}{}
\end{defn}

Some of the following definitions and theorems for standard systems can be given for descriptor systems, too. However, we restricted ourselves to the specializations we need here.

\begin{defn}[{\cite[Definition 2-6.1]{Dai1989}}]
Let $S=(E,A,B,C,D)\in\sys{n,p,m}$. Then
\[\Fkt{S}\colon\rho(E,A)\ni s\longmapsto C(sE-A)^{-1}B+D\in\mb{R}^{p\times m}\]
 is called the \indx{transfer function} of $S$. We denote the set of all $p\times m$ transfer functions by $\rat{p}{m}:=\{\Fkt{S};~S\in\bigcup_{n\in\mb{N}}\sys{n,p,m}\}$. A system $\wt{S}\in\sys{\wt{n},p,m}$ is called a \indx{realization} of $\Fkt{S}$, if $\Fkt{S}(s)=\Fkt{\wt{S}}(s)$ for all $s\in \dom(\Fkt{S})\cap\dom(\Fkt{\wt{S}})$.  We denote the set of all realizations of $\Fkt{S}$ by $\rez{\Fkt{S}}$. A realization $S_1\in\sys{n_1,p,m}$ of $\Fkt{S}$ is said to be \indx{minimal}, if every other realization $S_2\in\sys{n_2,p,m}$ of $\Fkt{S}$ fulfils $n_2\geq n_1$. 
 \nomenclature[A]{$\rat{p}{m}$}{}
\end{defn}

\begin{rem}\label{rem:transop} Transfer functions are meromorphic. Obviously, $\Fkt{S_1\oplus S_2}=\Fkt{S_1}+\Fkt{S_2}$ and $\Fkt{\rse{P}{S}{Q}}=\Fkt{S}$ hold. 
\end{rem}

We want to transform a system $S=(E,A,B,C,D)$ into $\wt{S}=(\wt{E},\wt{A},\wt{B},\wt{C},\wt{D})$ with $\Fkt{S}=\Fkt{\wt{S}}$, in particular $\sigma(E,A)=\sigma(\wt{E},\wt{A})$, such that $\wt{E},\wt{A},\wt{B},\wt{C},\wt{D}$ have a special structure, e.g.\  to decompose systems into a stable and antistable part. Note that all matrices stay real under the presented transformations.

\begin{defn}[{\cite[Definition 1-3.1]{Dai1989}}]
Two systems $S_i=(E_i,A_i,B_i,C_i,D_i)\in\sys{n,p,m}$, $i\in\{1,2\}$, are called \subindx{restricted system equivalent}{System} (for short r.s.e., $S_1\sim S_2$), if there exist regular $P,Q\in\mb{R}^{n\times n}$ such that $\rse{P}{S_1}{Q}=S_2$.
\end{defn}



\begin{lem}\label{lem:pspip}
For every $S=(E,A,B,C,D)\in\sys{n,p,m}$ there exist $S_w\sim S$, $k\in\{0,\ldots,n\}$ and $\nu\in\mb{N}$ such that 
\[S_w=\left(\bpm I_k&0\cr 0&N\epm,\bpm J&0\cr 0&I_{n-k}\epm,\bpm B_J\\B_N\epm,\bpm C_J&C_N\epm,D\right),\]
and $N^\nu=0$. Moreover, $\dom(\Fkt{S})=\rho(J)=\rho(E,A)$ holds and
\begin{EQ}\label{eq:ratpoly}
\forall\, s\in \dom(\Fkt{S})\colon \Fkt{S}(s)=C_J(sI-J)^{-1}B_J+D-\sum_{i=0}^{\nu-1}s^iC_NN^iB_N\,.
\end{EQ}
\end{lem}
\begin{proof}
 We apply \autoref{thm:decsyl} to $(E,A)$ with $M_1=\sigma(E,A)\backslash\{\infty\}$, $M_2=\{\infty\}$ and get regular $E_1,A_3,P,Q$ as in the theorem. Choose $\wt{P}:=\bsm E_1^{-1}&0\\0&I\esm P$, $\wt{Q}:=\bsm I&0\\0&A_3^{-1}\esm Q$ and set $S_w:=\rse{\wt{P}}{S}{\wt{Q}}$. The matrix $N:=E_3$ is nilpotent because $\sigma(N,I)=\{\infty\}$ and thus $\sigma(N)=\{0\}$. Equation \eqref{eq:ratpoly} follows from $\Fkt{S}=\Fkt{S_w}$ (\autoref{rem:transop}) and $C_N(sN-I)^{-1}B_N=-\sum_{i=0}^{\nu-1}s^iC_NN^iB_N$ (Neumann series).
\end{proof}

\begin{thm}[{\cite[Theorem 3.10, Theorem 3.17]{Zhou1996}}]\label{thm:uetrafmin2}\label{thm:minaeq}
Let $S=(I,A,B,C,D)\in\sys{n,p,m}$. Then
\begin{EQ}
S\sim\left(I_n,\bpm A_{11}&0&A_{13}&0\\A_{21}&A_{22}&A_{23}&A_{24}\\0&0&A_{33}&0\\0&0&A_{43}&A_{44}\epm,\bpm B_1\\B_2\\0\\0\epm,\bpm C_1&0&C_3&0\epm,D\right),
\end{EQ}
where $A_{ii}\in\mb{R}^{n_i\times n_i}$, $n_i\in\{0,\ldots,n\}$ for $i\in\{1,2,3,4\}$ and $S_M:=(I,A_{11},B_1,C_1,D)$ is a minimal realization of $\Fkt{S}$. This decomposition is known as the Kalman decomposition. All minimal realizations of $\Fkt{S}$ are r.s.e.\ to $S_M$.
\end{thm}

Now we introduce the controllability and observability Gramians. Later these matrices will play a central role in our problem of computing an optimal $RH_\infty$-approximation. The next theorem follows easily from \cite[Corollary 2]{Penzl1998}. 

\begin{thm}[{\cite{Penzl1998}}, Lyapunov equation]\label{thm:lyap}
Let $S=(E,A,B,C,D)\in\sys{n,p,m}^-$. Then there exist $X_c,\,X_o\in\mb{R}^{n\times n}$ such that \index{Lyapunov equation}
\begin{EQ}[rl]\label{eq:lyap}
AX_cE^{\top}+EX_cA^{\top}+BB^{\top}&=0\,,\\
A^{\top}X_oE+E^{\top}X_oA+C^{\top}C&=0\,.
\end{EQ}
They are unique and symmetric. We denote them by $\Cont{S}:=X_c$ (\indx{controllability Gramian} of system $S$) and $\Obs{S}:=X_o$ (\indx{observability Gramian}). 
\end{thm}

\begin{lem}\label{lem:XYaeq} 
Let $S=(E,A,B,C,D)\in\sys{n,p,m}^+$ and $E$ be regular. Let $P,Q\in\mb{R}^{n\times n}$ be regular. Then $\Cont{S}=Q\Cont{\rse{P}{S}{Q}}Q^{\top} \text{ and }\Obs{S}=P^{\top}\Obs{\rse{P}{S}{Q}}P$ hold. 
\end{lem}
\begin{proof}
We prove this only for $\Cont{S}$ since $\Obs{S}$ is analogously. By definition of $\Cont{\rse{P}{S}{Q}}$ 
\begin{EQ}[rl]
 0&=PAQ\Cont{\rse{P}{S}{Q}}Q^{\top}E^{\top}P^{\top}+PEQ\Cont{\rse{P}{S}{Q}}Q^{\top}A^{\top}P^{\top}+PBB^{\top}P^{\top}\,,\\
  &=P\biggl(A(Q\Cont{\rse{P}{S}{Q}}Q^{\top})E^{\top}+E(Q\Cont{\rse{P}{S}{Q}}Q^{\top})A^{\top}+BB^{\top}\biggr)P^{\top}
\end{EQ}
hold. Due to the regularity of $P$ and $Q$ and the uniqueness of the solution of \eqref{eq:lyap} (\autoref{thm:lyap}) the assertion is proved.
\end{proof}

\begin{thm}[{\cite[Theorem 4.3]{Glover1984}}]\label{thm:bal}
Let $S=(I,A,B,C,D)\in\sys{n,p,m}^-$. Then there exist $S_b\sim S$, $r\in\mb{N}$, $h\in\mb{R}_{\geq0}$ such that
\begin{align*}
 \Cont{S_b}=\bpm \Sigma_1&0\\0&-h I_r\epm,&&\Obs{S_b}=\bpm \Sigma_2&0\\0&-h I_r\epm,&&h^2\notin\sigma(\Sigma_1\Sigma_2)\,.
\end{align*}
\end{thm}
If $\Sigma_1=\Sigma_2$ are diagonal, then $S_b$ is known as the \indx{balanced realization} of $\Fkt{S}$. 


To check whether $\lambda\in\mb{C}$ is an eigenvalue of $\penc{E,A}$, we need the following easy generalization of a known result for standard systems (e.g.\  {\cite[Theorem 5.6.1 c]{Mackenroth2004}}).

\begin{lem}\label{lem:pol2}
Let $G\in\rat{p}{m}$ have a realization $S=(E,A,B,C,D)\in\sys{n,p,m}$. Then for every finite $\lambda\in\sigma\!\penc{E,A}$ at least one of the following three statements holds:
\begin{flalign*}
 (i)~\lambda \text{ is a pole of } G,&&(ii)~\rank\bpm \lambda E-A&B\epm<n,&&(iii)~\rank\bpm \lambda E-A\\C\epm<n\,.
\end{flalign*}
\end{lem}
\begin{proof}
Let $S_w\sim S$ be as in \autoref{lem:pspip}. We set $S_J:=(I,J,B_J,C_J,D)(\in\sys{k,p,m})$. By \autoref{lem:pspip} we know that $\lambda\in\sigma(E,A)\backslash\left\{\infty\right\}$ $\Leftrightarrow$ $\lambda\in\sigma(J)$. By \cite[Theorem 5.6.1 c]{Mackenroth2004} one of the following three statements holds:
\begin{flalign*}
 (i')~\lambda \text{ is a pole of } \Fkt{S_J},&&(ii')~\rank\bpm \lambda I-J&B_J\epm<k,&&(iii')~\rank\bpm \lambda I-J\\C_J\epm<k\,.
\end{flalign*}
Equation \eqref{eq:ratpoly} shows that $(i)\Leftrightarrow(i')$. By the regularity of $\lambda N-I$ we have
\[\rank\bpm \lambda I-J&B_J\epm+(n-k)=\rank\bpm \lambda I-J&0&B_J\cr 0&\lambda N-I&B_N\epm=\rank\bpm \lambda E-A&B\epm\]
and analogously $\rank\bsm \lambda I-J\\C_J\esm+(n-k)=\rank\bsm \lambda E-A\\C\esm$. Thus $(ii)\Leftrightarrow(ii')$ and $(iii)\Leftrightarrow(iii')$. This finishes the proof.
\end{proof}

\begin{defn}[{\cite[Section 4.3]{Zhou1996}}]\label{d_rloo}We approximate unstable systems with respect to the following spaces and corresponding norms. Let $p,m\in\mb{N}$.
\begin{itemize}
 \item $RL_\infty^{p\times m}:=\left\{G\in\rat{p}{m};~\ie\mb{R}\subseteq\dom(G),~\left\|G\right\|_\infty<\infty\right\}$, $\left\|G\right\|_\infty:=\sup\limits_{\omega\in \mb{R}}\left\|G(\ie\omega)\right\|_2$;
 \item $RH_\infty^{p\times m}:=\left\{G\in RL_\infty^{p\times m};~G \text{ analytic on } \mb{C}_{>0}\right\}$;
 \item $RL_2^{p\times m}:=\left\{G\in\rat{p}{m};~\ie\mb{R}\subseteq\dom(G),~\left\|G\right\|_2<\infty\right\}$, $\left\|G\right\|_2:=\sqrt{\frac{1}{2\pi}\int\limits_{-\infty}^\infty\|G(\ie\omega)\|_F^2d\omega}$;
  \item $RH_2^{p\times m}:=\left\{G\in RL_2^{p\times m};~G \text{ analytic on } \mb{C}_{>0}\right\}$.
\end{itemize}
We omit the indices, if there is no risk of confusion.
\end{defn}


\begin{rem}
These spaces are often introduced as a subset of socalled real rational functions, i.e.\  entrywise a quotient of polynomials with real coefficients. Actually, every real rational function is a transfer function (see \cite[Theorem 2-6.3]{Dai1989}) and vice versa, since $G(s)_{i,j}=e_i^{\top}C(sE-A)^{-1}Be_j+D_{i,j}$ is the Schur complement of $sE-A$ in $M(s):=\bsm sE-A&Be_j\\e_i^{\top}C&D_{i,j}\esm$ and thus $G(s)_{i,j}=\det(G(s)_{i,j})=\frac{\det(M(s))}{\det(sE-A)}$ by \autoref{thm:schur}. 
Let $F$ and $G$ be elements of one of these spaces with $\rez{F}=\rez{G}$, e.g.\  $s\mapsto \frac{1}{s-1}$ and $s\mapsto \frac{s-2}{(s-1)(s-2)}$. By definition they are equal on the cofinite set $\dom(F)\cap\dom(G)$. Therefore we identify them as the same function. With this and by the identity theorem for analytic functions the spaces are indeed normed spaces. \end{rem}

\begin{thm}[{\cite[Section 2.3]{Francis1987}}]\label{thm:orth} 
Let $G_-\in RL_2$ be analytic on $\mb{C}_{<0}$ and $G_+\in RH_2$. Then $\|G_-+G_+\|^2_2=\|G_-\|^2_2+\|G_+\|^2_2$ holds.
\end{thm}

The following known result characterizes these spaces in terms of realizations of transfer functions.

\begin{thm}[{\cite[Section 2.3]{Francis1987}}]\label{thm:norm} 
Let $G\in\rat{p}{m}$ with $\ie\mb{R}\subseteq\dom(G)$. Then the following hold:

\vspace{-0.2cm}
\hspace{-2.8cm}
\begin{minipage}{\textwidth}\begin{EQ}[rrll]
&\text{(a)}~&G\in RL_\infty & \Leftrightarrow \exists\, (I,A,B,C,D)\in\rez{G},\\
&\text{(b)}~&G\in RH_\infty & \Leftrightarrow \exists\, (I,A,B,C,D)\in\rez{G}\cap\sys{n,p,m}^+\,,\\
&\text{(c)}~&G\in RL_2 & \Leftrightarrow \exists\, (I,A,B,C,0)\in\rez{G},\\
&\text{(d)}~&G\in RH_2 & \Leftrightarrow \exists\, (I,A,B,C,0)\in\rez{G}\cap\sys{n,p,m}^+\,.
\end{EQ}
\end{minipage}
\vspace{0.2cm}
\end{thm}
%
%


\section{Optimal stable approximations}\label{sec:appr}
\subsection{Problem statement}

In this paper, we deal with the following problem. 

\fbox{
\begin{minipage}{0.97\textwidth}
\begin{prob} Let $q\in\{2,\infty\}$. For $S\in\sys{n,p,m}^0$ find $\wt{S}\in\bigcup_{\widehat{n}\in\mb{N}}\sys{\widehat{n},p,m}^+$ such that \[\|\Fkt{S}-\Fkt{\wt{S}}\|_q=\infb{\|\Fkt{S}-\Fkt{\widehat{S}}\|_q}{\widehat{S}\in\bigcup_{\widehat{n}\in\mb{N}}\sys{\widehat{n},p,m}^+}\,.\]\label{p_p1} \label{p_p2}
\end{prob}\end{minipage}}

In the next subsections we show that this problem is solvable. We also present an explicit solution in \autoref{thm:l2min} for $q=2$ and \autoref{thm:loomin} for $q=\infty$.

We do not consider systems $S=(E,A,B,C,D)\in\sys{n,p,m}$ where $\penc{E,A}$ has imaginary eigenvalues. If there is no imaginary pole of $\Fkt{S}$, then there exists a realization $S_2\in\sys{n,p,m}^0$ of $\Fkt{S}$ and we solve \hyperref[p_p1]{(AP$_q$)} for $S_2$. Otherwise, for every $\widehat{S}\in\bigcup_{\widehat{n}\in\mb{N}}\sys{\widehat{n},p,m}^+$ we have $\|\Fkt{S}-\Fkt{\widehat{S}}\|_q=\infty$ and hence \hyperref[p_p1]{(AP$_q$)} is not solvable.

First, we show that solving \hyperref[p_p1]{(AP$_q$)} for a descriptor system we can equivalently solve \hyperref[p_p1]{(AP$_q$)} for its antistable part. This was e.g.\ proposed in \cite[Section 8.3]{Francis1987} but only for standard systems and $q=\infty$. We need \autoref{thm:stabinst} to use some results in \cite{Glover1984} and \cite{Safonov1990} which are only applicable for antistable standard systems. 

\begin{thm}\label{thm:stabinst}
Let $S\in\sys{n,p,m}^0$. Then there exist $S_+=(E_+,A_+,B_+,C_+,D)\in\sys{n_+,p,m}^+$ and $S_-=(E_-,A_-,B_-,C_-,0)\in\sys{n_-,p,m}^-$ such that $S\sim S_+\oplus S_-$. Let $q\in\{2,\infty\}$ and $\gamma\geq0$. Then the following two are equivalent:
\begin{flalign*}
(i)~&\exists\, \wt{S}\in\bigcup_{\widehat{n}\in\mb{N}}\sys{\widehat{n},p,m}^+\colon\|\Fkt{S}-\Fkt{\wt{S}}\|_q\leq\gamma\,,&
(ii)~&\exists\, \wt{S}\in\bigcup_{\widehat{n}\in\mb{N}}\sys{\widehat{n},p,m}^+\colon\|\Fkt{S_-}-\Fkt{\wt{S}}\|_q\leq\gamma\,.
\end{flalign*}
If $\wt{S}$ satisfies $(ii)$, then $\wt{S}\oplus S_+$ satisfies $(i)$.
\end{thm}
\begin{proof}
\autoref{thm:decsyl} assures the existence of such systems $S_+$ and $S_-$.
Let $\wt{S}\in\sys{\wt{n},p,m}^+$ satisfy $(i)$. Then we have
\[\gamma\geq\|\Fkt{S}-\Fkt{\wt{S}}\|_q=\|\Fkt{S_-}-(\Fkt{\wt{S}}-\Fkt{S_+})\|_q\,.\]

Thus $\widehat{S}:=\wt{S}\oplus(E_+,A_+,B_+,-C_+,-D)$ satisfies $(ii)$ (compare \autoref{rem:transop}). The fact that $\widehat{S}\in\bigcup_{\widehat{n}\in\mb{N}}\sys{\widehat{n},p,m}^+$ follows easily from the definition of $\oplus$. Let $\wt{S}\in\sys{\wt{n},p,m}^+$ fulfils $(ii)$. Then analogously $\wt{S}\oplus S_+$ satisfies $(i)$. 
\end{proof}

\subsection{Optimal \texorpdfstring{{$\boldsymbol{RH_2}$}}{RH{\tiny{2}}}-approximations}


\begin{thm}\label{thm:l2min}
Let $S\in\sys{n,p,m}^0$. We use the notations of \autoref{thm:stabinst}. Then $S_+$ solves \hyperref[p_p1]{(AP$_2$)}. More precisely, we have
\[\infb{\|\Fkt{S}-\Fkt{\widehat{S}}\|_2}{\widehat{S}\in\bigcup_{\widehat{n}\in\mb{N}}\sys{\widehat{n},p,m}^+}=\|\Fkt{S}-\Fkt{S_+}\|_2=\|\Fkt{S_-}\|_2\,.\]
The solution is unique in the following sense: If $S_2$ is another solution, then it is also a realization of $\Fkt{S_+}$.
\end{thm}
\begin{proof}
Let $\widehat{S}\in\bigcup_{\widehat{n}\in\mb{N}}\sys{\widehat{n},p,m}^+$ with $\Fkt{S}-\Fkt{\widehat{S}}\in RL_2$. Since $\Fkt{S_-}\in RL_2$ (\autoref{thm:norm}) we have $\Fkt{S_+}-\Fkt{\widehat{S}}\in RH_2$. By \autoref{thm:orth} we obtain
\[\|\Fkt{S}-\Fkt{\widehat{S}}\|_2^2=\|\Fkt{S_-}+\Fkt{S_+}-\Fkt{\widehat{S}}\|_2^2=\|\Fkt{S_-}\|_2^2+\|\Fkt{S_+}-\Fkt{\widehat{S}}\|_2^2\geq\|\Fkt{S_-}\|_2^2\,.\]
The lower bound is attained, if and only if $\widehat{S}\in\rez{\Fkt{S_+}}$, e.g.\ if $\widehat{S}=S_+$.
\end{proof}

\begin{rem}
In the proof of \autoref{thm:l2min} it can be seen that the function $\Fkt{S_+}$ is even the best approximation in the Hardy space $H_2$. The reason is that $\Fkt{S_-}$ is also orthogonal to $H_2$ (see the more general statement of \autoref{thm:orth} in \cite[Section 2.3]{Francis1987}).
\end{rem}

\subsection{Optimal \texorpdfstring{{$\boldsymbol{RH_\infty}$}}{RH{\tiny{oo}}}-approximations with balanced realization}


For a minimal antistable standard system $S$ the problem \hyperref[p_p2]{(AP$_\infty$)} was already solved in \cite{Glover1984}. Combining {\cite[Theorem 6.1]{Glover1984}} and its proof based on \cite[Theorem 6.3]{Glover1984}, we obtain the following result. We reformulate the statements in terms of realizations instead of transfer functions. Note also the interchange of the stable and antistable system which is easy to show.

\begin{thm}[{\cite{Glover1984}}]\label{thm:enehari}\label{thm:optb}
Let $S\in\sys{n,p,m}^-$ be a minimal realization such that
\begin{align*}
S&=\left(\bpm I&0\\0&I_r\epm,\bpm A_{11}&A_{12}\\A_{21}&A_{22}\epm,\bpm B_1\\B_2\epm,\bpm C_1&C_2\epm,D\right),\\
 \Cont{S}&=\bpm \Sigma_1&0\\0&-\sigma_1 I_r\epm,\quad\Obs{S}=\bpm \Sigma_2&0\\0&-\sigma_1 I_r\epm,\quad\sigma_1^2\notin\sigma(\Sigma_1\Sigma_2)\,,
\end{align*}
where $\sigma_1:=\sqrt{\max\sigma(\Cont{S}\Obs{S})}$ (compare \autoref{thm:bal}). 
Then there exists $U\in\mb{R}^{p\times m}$ with $B_2=-C_2^{\top}U$. Set 
\begin{EQ}[rlcrl]
\Gamma&:=\Sigma_1\Sigma_2-\sigma_1^2I\,,\\
\wt{A}&:=\Gamma^{-1}(\sigma_1^2A_{11}^{\top}+\Sigma_2A_{11}\Sigma_1+\sigma_1C_1^{\top}UB_1^{\top})\,,&~~&\wt{C}&:=C_1\Sigma_1-\sigma_1UB_1^{\top},\\
\wt{B}&:=\Gamma^{-1}(\Sigma_2B_1-\sigma_1C_1^{\top}U)\,,&~~&\wt{D}&:=D+\sigma_1U\,.
\end{EQ}
Then $\wt{S}:=(I,\wt{A},\wt{B},\wt{C},\wt{D})\in\sys{n-r,p,m}^+$ holds and 
\[\infb{\|\Fkt{S}-\Fkt{\widehat{S}}\|_\infty}{\widehat{S}\in\bigcup_{\widehat{n}\in\mb{N}}\sys{\widehat{n},p,m}^+}=\|\Fkt{S}-\Fkt{\wt{S}}\|_\infty=\sigma_1\,.\]
\end{thm}

\begin{rem}
The value $\sigma_1$ is the greatest \indx{hankel singular value} of $\Fkt{S}$ (see \cite[Section 5.1]{Francis1987} for further details). As stated in \cite[Theorem 6.1]{Glover1984}, the function $\Fkt{\wt{S}}$ is even the best approximation in the more general Hardy space $H_\infty$. 
\end{rem}

Because of \autoref{thm:stabinst} we are now able to solve \hyperref[p_p2]{(AP$_\infty$)}, if $S$ is a descriptor system. The disadvantage of algorithms based on \autoref{thm:optb} is that we need a balanced minimal realization of $\Fkt{S_-}$ to compute an optimal solution. 

\subsection{Optimal \texorpdfstring{{$\boldsymbol{RH_\infty}$}}{RH{\tiny{oo}}}-approximations without balanced realization}

Some results of \cite{Glover1984}, concerning optimal hankel norm approximation, were generalized in \cite{Safonov1990} to nonminimal stable standard systems. We use \cite[Theorem 1]{Safonov1990} and its proof to solve \hyperref[p_p2]{(AP$_\infty$)} for general $S\in\sys{n,p,m}^-$. Among some modifications, one of our main tasks is to prove that the optimal solution lies indeed in $\sys{n,p,m}^+$, since \cite{Safonov1990} only considers the poles of the resulting transfer function.

First, we need to compute the greatest hankel singular value $\sigma_1$ of $\Fkt{S}$ with the formula given in \autoref{thm:enehari} which is restricted to minimal standard systems. However, we avoid the determination of a minimal realization. To the author's knowledge, it was never shown that the nonzero hankel singular values of $\Fkt{S}$ equal the roots of the nonzero eigenvalues of $\Cont{S}\Obs{S}$ independently of the realization $S$ of $\Fkt{S}$. This equality was only proved for minimal realizations (cf.\ e.g.\ \cite[Section 5.1, Theorem 3]{Francis1987}).

\begin{lem}\label{lem:hank}
Let $S=(I,A,B,C,D)\in\sys{n,p,m}^-$ be a minimal realization of $G\in \rat{p}{m}$ and let $\wt{S}=(\wt{E},\wt{A},\wt{B},\wt{C},\wt{D})\in\sys{\wt{n},p,m}^-$ be a realization of $G$. Then
\[\sigma\bigl(\Cont{S}\Obs{S}\bigr)\backslash\left\{0\right\}=\sigma\bigl(\Cont{\wt{S}}\wt{E}^{\top}\Obs{\wt{S}}\wt{E}\bigr)\backslash\left\{0\right\}.\]
In particular, the equation $\sigma_1^2=\max\sigma\bigl(\Cont{\wt{S}}\wt{E}^{\top}\Obs{\wt{S}}\wt{E}\bigr)$ holds for all $\wt{S}\in\sys{\wt{n},p,m}^-\cap\rez{G}$.
\end{lem}
\begin{proof}
We set $S_I:=\rser{\wt{E}^{-1}}{\wt{S}}$. By \autoref{thm:uetrafmin2} there exists a regular $T\in\mb{R}^{\wt{n}\times \wt{n}}$ such that
\begin{EQ}[rl]
S_T&:=\rse{T}{S_I}{T^{-1}}
=\left(I,\bpm A_{11}&0&A_{13}&0\\A_{21}&A_{22}&A_{23}&A_{24}\\0&0&A_{33}&0\\0&0&A_{43}&A_{44}\epm,\bpm B_1\\B_2\\0\\0\epm,\bpm C_1&0&C_3&0\epm,\wt{D}\right),
\end{EQ}
where $S_1:=(I,A_{11},B_1,C_1,\wt{D})$ is a minimal realization of $G$. The condition $\sigma(\wt{E},\wt{A})=\sigma(\wt{E}^{-1}\wt{A})\subseteq\mb{C}_{>0}$ implies
\begin{align*}
 \sigma\bpm A_{11}&0\\A_{21}&A_{22}\epm\subseteq\mb{C}_{>0}\,,&&\sigma\bpm A_{11}&A_{13}\\0&A_{33}\epm\subseteq\mb{C}_{>0}\,,&&\sigma( A_{11})\subseteq\mb{C}_{>0}\,.
\end{align*}
By \autoref{thm:lyap} there exist unique solutions $\bsm X_1&X_2\\X_2^{\top}&X_4\esm$ and $\bsm Y_1&Y_2\\Y_2^{\top}&Y_4\esm$ of
\begin{EQ}[rlll]
 \bpm A_{11}&0\\A_{21}&A_{22}\epm \bpm X_1&X_2\\X_2^{\top}&X_4\epm&+\bpm X_1&X_2\\X_2^{\top}&X_4\epm\bpm A_{11}&0\\A_{21}&A_{22}\epm^{\top}&+\bpm B_1\\B_2\epm\bpm B_1^{\top}&B_2^{\top}\epm&=0\,,\\
\bpm A_{11}&A_{13}\\0&A_{33}\epm^{\top} \bpm Y_1&Y_2\\Y_2^{\top}&Y_4\epm&+\bpm Y_1&Y_2\\Y_2^{\top}&Y_4\epm\bpm A_{11}&A_{13}\\0&A_{33}\epm&+\bpm C_1^{\top}\\C_3^{\top}\epm\bpm C_1&C_3\epm&=0\,.
\end{EQ}
Thus $X_1=\Cont{S_1}$, $Y_1=\Obs{S_1}$,
\[\Cont{S_T}=\bpm X_1&X_2&0&0\\X_2^{\top}&X_4&0&0\\0&0&0&0\\0&0&0&0\epm \text{ and }\Obs{S_T}=\bpm Y_1&0&Y_2&0\\0&0&0&0\\Y_2^{\top}&0&Y_4&0\\0&0&0&0\epm.\]
By \autoref{thm:minaeq} and \autoref{lem:XYaeq} there exists a regular $P\in\mb{R}^{n\times n}$ with $\Cont{S_1}=P^{-1}\Cont{S}P^{-\top}$, $\Obs{S_1}=P^{\top}\Obs{S}P$.
We conclude
\begin{align*}
 \sigma\bigl(\Cont{\wt{S}}\wt{E}^{\top}\Obs{\wt{S}}\wt{E}\bigr)\backslash\left\{0\right\}&=\sigma\bigl(\Cont{S_I}\Obs{S_I}\bigr)\backslash\left\{0\right\}=\sigma\bigl(T\Cont{S_T}T^{\top}T^{-\top}\Obs{S_T}T^{-1}\bigr)\backslash\left\{0\right\}\\
&=\sigma\bigl(\Cont{S_T}\Obs{S_T}\bigr)\backslash\left\{0\right\}=\sigma\bpm X_1Y_1&0&X_1Y_2&0\\X_2^{\top}Y_1&0&X_2^{\top}Y_2&0\\0&0&0&0\\0&0&0&0\epm\backslash\left\{0\right\}\\
&=\sigma\bigl(X_1Y_1\bigr)\backslash\left\{0\right\}=\sigma\bigl(\Cont{S_1}\Obs{S_1}\bigr)\backslash\left\{0\right\}
=\sigma\bigl(\Cont{S}\Obs{S}\bigr)\backslash\left\{0\right\}\,. \qedhere 
\end{align*}
\end{proof}

The next theorem is one of our main results. Together with \autoref{thm:stabinst} it gives suboptimal solutions of \hyperref[p_p2]{(AP$_\infty$)} and, with an additional regularity-condition, an optimal solution. \autoref{thm:sing3} solves \hyperref[p_p2]{(AP$_\infty$)}, if the regularity-condition is not fulfilled.

\begin{thm}\label{thm:causal2}\label{thm:causal3}
Let $S=(E,A,B,C,D)\in\sys{n,p,m}^-$. Choose $\gamma\geq\sigma_1$ and set
\begin{EQ}[rlcrlcrl]
\opt{R}{S}{\gamma}&:=\Obs{S}E\Cont{S}E^{\top}-\gamma^2I\,,&~&\opt{E}{S}{\gamma}&:=E^{\top}\opt{R}{S}{\gamma}\,,&~&\opt{B}{S}{\gamma}&:=E^{\top}\Obs{S}B,\\
&&~&\opt{A}{S}{\gamma}&:=-A^{\top}\opt{R}{S}{\gamma}-C^{\top}\opt{C}{S}{\gamma}\,,&~~~~&\opt{C}{S}{\gamma}&:=C\Cont{S}E^{\top}.
\end{EQ}
If $\penc{\opt{E}{S}{\gamma},\opt{A}{S}{\gamma}}$ is regular, e.g.\  if $\gamma>\sigma_1$, then $\opt{\mc{S}}{S}{\gamma}:=(\opt{E}{S}{\gamma},\opt{A}{S}{\gamma},\opt{B}{S}{\gamma},\opt{C}{S}{\gamma},D)\in\sys{n,p,m}^+$ and $\opt{\mc{S}}{S}{\gamma}$ fulfils
\[\sigma_1\leq\left\|\Fkt{S}-\Fkt{\opt{\mc{S}}{S}{\gamma}}\right\|_\infty\leq\gamma\,.\]
\end{thm}
\begin{proof}
First, let $E=I$. In this and the following proofs we use the equation
\begin{EQ}\label{eq:asg}
 \opt{A}{S}{\gamma}=-A^{\top}\opt{R}{S}{\gamma}-C^{\top}\opt{C}{S}{\gamma}=-\opt{R}{S}{\gamma}A^{\top}-\opt{B}{S}{\gamma}B^{\top}=\gamma^2A^{\top}+\Obs{S}A\Cont{S},
\end{EQ}
which holds because of
\begin{EQ}[rl]
-A^{\top}\opt{R}{S}{\gamma}-C^{\top}\opt{C}{S}{\gamma}&=\gamma^2A^{\top}-A^{\top}\Obs{S}\Cont{S}-C^{\top}C\Cont{S}=\gamma^2A^{\top}+\Obs{S}A\Cont{S}\\
&=\gamma^2A^{\top}-\Obs{S}\Cont{S}A^{\top}-\Obs{S}BB^{\top}=-\opt{R}{S}{\gamma}A^{\top}-\opt{B}{S}{\gamma}B^{\top}.
\end{EQ}
Thus $\opt{A}{S}{\gamma}$ is indeed the same as in \cite[Theorem 1]{Safonov1990}. To apply \cite[Theorem 1]{Safonov1990}, we first construct an antistable system $S_2$ with a hankel singular value greater than $\gamma$. This method was proposed in \cite[Remark 8.4]{Glover1984}. Let $\gamma_1>\gamma$. We set $S_1:=(1,\tfrac12\gamma_1,\gamma_1,\gamma_1,0)$ and
\[S_2:=\left(I,\bpm A&0\\0&\tfrac12\gamma_1\epm,\bpm B&0\\0&\gamma_1\epm,\bpm C&0\\0&\gamma_1\epm,\bpm D&0\\0&0\epm\right)\in\sys{n+1,p+1,m+1}^-\,.\]
Then we have
\[\Cont{S_2}=\bpm \Cont{S}&0\\0&-\gamma_1\epm,~~~~\Obs{S_2}=\bpm \Obs{S}&0\\0&-\gamma_1\epm.\]
Thus $\gamma_1$ is the greatest hankel singular value of $\Fkt{S_2}$ and $\sigma_1\leq\gamma<\gamma_1$. If we apply \cite[Theorem 1]{Safonov1990} to $S_2$ with $K=0\in RH_\infty$, then the transfer function of
\[\opt{\mc{S}}{S_2}{\gamma}=\left(\bpm \opt{E}{S}{\gamma}&0\\0&\gamma_1^2-\gamma^2\epm\!\!,\bpm \opt{A}{S}{\gamma}&0\\0&\tfrac12(\gamma^2\gamma_1+\gamma_1^3)\epm\!\!,\bpm \opt{B}{S}{\gamma}&0\\0&-\gamma_1^2\epm\!\!,\bpm \opt{C}{S}{\gamma}&0\\0&-\gamma_1^2\epm\!\!,\bpm D&0\\0&0\epm\right)\]
is a suboptimal solution, i.e.\ $\left\|\Fkt{S_2}-\Fkt{\opt{\mc{S}}{S_2}{\gamma}}\right\|_\infty\leq\gamma$. Note that \cite[Theorem 1]{Safonov1990} does not state, that for every $K\in RH_\infty$ with $\|K\|_\infty\leq1$ the resulting transfer function is a suboptimal solution. However, this was shown in the proof (see also \cite[Remark 8.4]{Glover1984}). With
\[\Fkt{S_2}=\bpm \Fkt{S}&0\\0&\Fkt{S_1}\epm,~~~~\Fkt{\opt{\mc{S}}{S_2}{\gamma}}=\bpm \Fkt{\opt{\mc{S}}{S}{\gamma}}&0\\0&\Fkt{\opt{\mc{S}}{S_1}{\gamma}}\epm\]
we conclude 
\[\left\|\Fkt{S}-\Fkt{\opt{\mc{S}}{S}{\gamma}}\right\|_\infty \leq\left\|\Fkt{S_2}-\Fkt{\opt{\mc{S}}{S_2}{\gamma}}\right\|_\infty\leq\gamma\,.\]
Now we show that $\opt{\mc{S}}{S}{\gamma}\in\sys{n,p,m}^+$. \cite[Theorem 1]{Safonov1990} states that $\Fkt{\opt{\mc{S}}{S_2}{\gamma}}$ has at most one pole $\lambda$ in $\mb{C}_{\geq0}$ counting multiplicities. The pole of $\Fkt{\opt{\mc{S}}{S_1}{\gamma}}$ is $\wt{\lambda}=\frac{2(\gamma_1^2-\gamma^2)}{\gamma^2\gamma_1+\gamma_1^3}>0$ with $\gamma_1^2-\gamma^2>0$, $\gamma^2\gamma_1+\gamma_1^3>0$. Thus $\lambda=\wt{\lambda}$ and $\Fkt{\opt{\mc{S}}{S}{\gamma}}$ has no poles in $\mb{C}_{\geq0}$. 

Let $\lambda\in\mb{C}_{\geq0}$. We use \autoref{lem:pol2} to show that $\lambda\notin\sigma(\opt{E}{S}{\gamma},\opt{A}{S}{\gamma})$. Let $v\in\mb{R}^n$ satisfy $(\lambda \opt{E}{S}{\gamma}-\opt{A}{S}{\gamma})v=0$ and $\opt{C}{S}{\gamma}v=0$. Hence with $\opt{A}{S}{\gamma}=-A^{\top}\opt{E}{S}{\gamma}-C^{\top}\opt{C}{S}{\gamma}$ we get $(\lambda  I-(-A^{\top}))\opt{E}{S}{\gamma}v=0$. Because of $\sigma(-A^{\top})\subseteq\mb{C}_{<0}$ we conclude $\opt{E}{S}{\gamma}v=0$. Thus
\[\forall\, s\in\mb{C}\colon (s\opt{E}{S}{\gamma}-\opt{A}{S}{\gamma})v=(s\opt{E}{S}{\gamma}+A^{\top}\opt{E}{S}{\gamma}+C^{\top}\opt{C}{S}{\gamma})v=0\,.\]
Since the matrix pair $\penc{\opt{E}{S}{\gamma},\opt{A}{S}{\gamma}}$ is regular we have $v=0$.

Now let $v\in\mb{R}^n$ satisfy $(\lambda \opt{E}{S}{\gamma}^{\top}-\opt{A}{S}{\gamma}^{\top})v=0$ and $\opt{B}{S}{\gamma}^{\top}v=0$.
Hence $(\lambda  I+A)\opt{E}{S}{\gamma}^{\top}v=0$ follows with $\opt{A}{S}{\gamma}^{\top}=-A\opt{E}{S}{\gamma}^{\top}-B\opt{B}{S}{\gamma}^{\top}$. Again we deduce $v=0$. Therefore $\lambda\notin\sigma\!\penc{\opt{E}{S}{\gamma},\opt{A}{S}{\gamma}}$ and thus $\opt{\mc{S}}{S}{\gamma}\in\sys{n,p,m}^+$ follows.

Finally we consider general regular $E\in\mb{R}^{n\times n}$. We already proved that the assertion holds for $\opt{\mc{S}}{\rser{E^{-1}}{S}}{\gamma}$. With $\Cont{\rser{E^{-1}}{S}}=\Cont{S}$ and $\Obs{\rser{E^{-1}}{S}}=E^{\top}\Obs{S}E$ (\autoref{lem:XYaeq}) we conclude
\begin{EQ}[rlcrl]
\opt{E}{\rser{E^{-1}}{S}}{\gamma}&=E^{\top}\Obs{S}E\Cont{S}-\gamma^2I\,,&~&\opt{B}{\rser{E^{-1}}{S}}{\gamma}&=E^{\top}\Obs{S}EE^{-1}B\,,\\
\opt{A}{\rser{E^{-1}}{S}}{\gamma}&=-A^{\top}E^{-\top}\opt{E}{\rser{E^{-1}}{S}}{\gamma}-C^{\top}\opt{C}{\rser{E^{-1}}{S}}{\gamma}\,,&~&\opt{C}{\rser{E^{-1}}{S}}{\gamma}&=C\Cont{S}\,.
\end{EQ}
Now we see that $\opt{\mc{S}}{S}{\gamma}=\rser{\opt{\mc{S}}{\rser{E^{-1}}{S}}{\gamma}}{E^{\top}}$. Thus the assertion also holds for $\opt{\mc{S}}{S}{\gamma}\sim \opt{\mc{S}}{\rser{E^{-1}}{S}}{\gamma}$.

If $\gamma>\sigma_1$, the matrix pair $\penc{\opt{E}{S}{\gamma},\opt{A}{S}{\gamma}}$ is regular, since $\opt{R}{S}{\gamma}$, $E$ and thus $\opt{E}{S}{\gamma}$ are regular.
\end{proof}

We formulate some equivalent conditions for the regularity-condition of $\penc{\opt{E}{S}{\sigma_1},\opt{A}{S}{\sigma_1}}$ in \autoref{thm:causal2}. 
\begin{cor}\label{c_singreg}
 The following three statements are equivalent.
\begin{flalign*}
 (i)~\penc{\opt{E}{S}{\sigma_1},\opt{A}{S}{\sigma_1}} \text{ is regular},&&(ii)~\sigma\!\penc{\opt{E}{S}{\sigma_1},\opt{A}{S}{\sigma_1}}\subseteq\mb{C}_{<0}\cup\{\infty\},&&(iii)~ \opt{A}{S}{\sigma_1} \text{ is regular}\,.
\end{flalign*}
\end{cor}
\begin{proof}
The implication ``(i) $\Rightarrow$ (ii)'' follows from \autoref{thm:causal3}. The statement (ii) implies $0\notin\sigma\!\penc{\opt{E}{S}{\sigma_1},\opt{A}{S}{\sigma_1}}$ and thus $(0\cdot \opt{E}{S}{\sigma_1}+\opt{A}{S}{\sigma_1})$ is regular. Finally, (iii) implies $0\notin\sigma\!\penc{\opt{E}{S}{\sigma_1},\opt{A}{S}{\sigma_1}}$, therefore $\rho\!\penc{\opt{E}{S}{\sigma_1},\opt{A}{S}{\sigma_1}}\neq\emptyset$ and thus (i) holds.
\end{proof}

To get optimal solutions if $\penc{\opt{E}{S}{\sigma_1},\opt{A}{S}{\sigma_1}}$ is singular, we analyze $\opt{\mc{S}}{S}{\gamma}$ in \autoref{thm:causal3} for $\gamma\to\sigma_1+$. For the proofs we use the realization $S_b\sim S$ of $\Fkt{S}$ as in \autoref{thm:bal} with $h=\sigma_1$ because we will take advantage of the special structure of $\opt{\mc{S}}{S_b}{\sigma_1}$. First we investigate the relation between $\opt{\mc{S}}{S}{\gamma}$ and $\opt{\mc{S}}{S_b}{\gamma}$.

\begin{lem}\label{lem:ballsg}
Let $S=(E,A,B,C,D)\in\sys{n,p,m}^-$. Let $P,Q\in\mb{R}^{n\times n}$ be regular. 
Then $\opt{\mc{S}}{\rse{P}{S}{Q}}{\gamma}=\rse{Q^{T}}{\opt{\mc{S}}{S}{\gamma}}{P^{T}}$ holds for all $\gamma\in\mb{R}$. 
\end{lem}
\begin{proof}
We set $\wt{S}:=\rse{P}{S}{Q}$. By \autoref{lem:XYaeq} the equations $\Cont{S}=Q\Cont{\wt{S}}Q^{\top}$, $\Obs{S}=P^{\top}\Obs{\wt{S}}P$ hold. This implies 
\begin{align*}
 \opt{B}{S}{\gamma}&=E^{\top}\Obs{S}B=E^{\top}P^{\top}\Obs{\wt{S}}PB=Q^{-\top}\opt{B}{\wt{S}}{\gamma}\,,\\
 \opt{C}{S}{\gamma}&=C\Cont{S}E^{\top}=CQ\Cont{\wt{S}}Q^{\top}E^{\top}=\opt{C}{\wt{S}}{\gamma}P^{-\top},\\
 \opt{E}{S}{\gamma}&=E^{\top}\Obs{S}E\Cont{S}E^{\top}-\gamma^2E^{\top}=E^{\top}P^{\top}\Obs{\wt{S}}PEQ\Cont{\wt{S}}Q^{\top}E^{\top}-\gamma^2E^{\top}=Q^{-\top}\opt{E}{\wt{S}}{\gamma}P^{-\top},\\
 \opt{A}{S}{\gamma}&=-A^{\top}E^{-\top}\opt{E}{S}{\gamma}+C^{\top}\opt{C}{S}{\gamma}=-A^{\top}E^{-\top}Q^{-\top}\opt{E}{\wt{S}}{\gamma}P^{-\top}+C^{\top}\opt{C}{\wt{S}}{\gamma}P^{-\top}\\
&=Q^{-\top}\opt{A}{\wt{S}}{\gamma}P^{-\top}.\qedhere
\end{align*}
\end{proof}

\begin{lem}\label{lem:sing}
Let $S_b=(I,A_b,B_b,C_b,D)\sim S$ be as in \autoref{thm:bal} with $h=\sigma_1$. Let $T_1\in\mb{R}^{n\times n}$ be regular such that
\[\rse{T_1}{S}{E^{-1}T_1^{-1}}=S_b=\left(\bpm I&0\\0&I_r\epm,\bpm A_{11}&A_{12}\\A_{21}&A_{22}\epm,\bpm B_1\\B_2\epm,\bpm C_1&C_2\epm,D\right).\]

Then the following three statements hold:
\begin{enumerate}[(a)]
 \item $B_2B_2^{\top}=C_2^{\top}C_2$,
 \item $\ker(B_2^{\top})=\ker(C_2)$,
 \item $\penc{\opt{E}{S}{\sigma_1},\opt{A}{S}{\sigma_1}}$ singular $\Leftrightarrow$ $\ker(C_2)\neq\left\{0\right\}$ $\Leftrightarrow$ $\ker(B_2^{\top})\neq\{0\}$.
\end{enumerate}
\end{lem}
\begin{proof}
We get $B_2B_2^{\top}=C_2^{\top}C_2$ by adding the equations $-\sigma_1 A_{22}-\sigma_1 A_{22}^{\top}+B_2B_2^{\top}=0$ and $\sigma_1 A_{22}^{\top}+\sigma_1 A_{22}-C_2^{\top}C_2=0$ resulting from \eqref{eq:lyap}. Therefore
\begin{EQ}[rl]
 v\in\ker(B_2^{\top})&\Leftrightarrow B_2^{\top}v=0 \Leftrightarrow \left\|B_2^{\top}v\right\|_2=0\Leftrightarrow 0=v^{\top}B_2B_2^{\top}v=v^{\top}C_2^{\top}C_2v\\
&\Leftrightarrow \left\|C_2v\right\|_2=0 \Leftrightarrow v\in\ker(C_2).
\end{EQ}

By \autoref{lem:ballsg} and with $\Gamma:=\Sigma_2\Sigma_1-\sigma_1^2I$ we have
\begin{EQ}[rl]
 \opt{\mc{S}}{S_b}{\sigma_1}&=\rse{T_1^{-\top}E^{-\top}}{\opt{\mc{S}}{S}{\sigma_1}}{T_1^{\top}}\\
 &=\left(\bpm \Gamma&0\\0&0\epm,\bpm -A_{11}^{\top}\Gamma-C_1^{\top}C_1\Sigma_1&\sigma_1C_1^{\top}C_2\\-A_{12}^{\top}\Gamma-C_2^{\top}C_1\Sigma_1&\sigma_1C_2^{\top}C_2\epm,\bpm\Sigma_2B_1\\-\sigma_1B_2\epm,\bpm C_1\Sigma_1&-\sigma_1 C_2\epm,0\right).
\end{EQ}
Therefore $\penc{\opt{E}{S}{\sigma_1},\opt{A}{S}{\sigma_1}}$ is singular, if and only if $\penc{\opt{E}{S_b}{\sigma_1},\opt{A}{S_b}{\sigma_1}}$ is singular.

Let $\ker(C_2)\neq\left\{0\right\}$. Then there exists $v\neq0$ with $C_2v=0$, i.e.\ for all $s\in\mb{C}$
\begin{EQ}[rl]
(s\opt{E}{S_b}{\sigma_1}-\opt{A}{S_b}{\sigma_1})\bpm 0\\v\epm=\bpm s\Gamma+A_{11}^{\top}\Gamma+C_1^{\top}C_1\Sigma_1&-\sigma_1C_1^{\top}C_2\\A_{12}^{\top}\Gamma+C_2^{\top}C_1\Sigma_1&-\sigma_1C_2^{\top}C_2\epm\bpm 0\\v\epm&=\bpm 0\\0\epm.
\end{EQ}
That means $\penc{\opt{E}{S_b}{\sigma_1},\opt{A}{S_b}{\sigma_1}}$ is singular.

Let $\ker(C_2)=\left\{0\right\}$. Then $C_2^{\top}C_2$ is regular. Let $s\in\mb{C}$. The matrix $s\opt{E}{S_b}{\sigma_1}-\opt{A}{S_b}{\sigma_1}$ is singular, if and only if the Schur complement
\[\biggl( sI+\biggl(A_{11}^{\top}+C_1^{\top}C_1\Sigma_1\Gamma^{-1}-C_1^{\top}C_2(C_2^{\top}C_2)^{-1}(A_{12}^{\top}\Gamma+C_2^{\top}C_1\Sigma_1)\Gamma^{-1}\biggr)\biggr)\Gamma=:(sI-M)\Gamma\]
of $s\opt{E}{S_b}{\sigma_1}-\opt{A}{S_b}{\sigma_1}$ in $-\sigma_1C_2^{\top}C_2$ is singular (\autoref{thm:schur}). This is equivalent to $s\in\sigma(M)$. Hence $\penc{\opt{E}{S}{\sigma_1},\opt{A}{S}{\sigma_1}}$ has a finite number of eigenvalues and consequently is regular.
\end{proof}

Similar to the proof of \cite[Theorem 1]{Safonov1990} we want to eliminate the singular part of the system $\opt{\mc{S}}{S}{\sigma_1}$ and show that the resulting system is an optimal solution.

\begin{lem}\label{lem:singreg}
 Let $\penc{\opt{E}{S}{\sigma_1},\opt{A}{S}{\sigma_1}}$ be singular. Then there exists a regular $T\in\mb{R}^{n\times n}$ such that
\[\rse{TE^{-\top}}{\opt{\mc{S}}{S}{\sigma_1}}{T^{-1}}=\left(\bpm \wt{E}_{11}&0\\0&0\epm,\bpm \wt{A}_{11}&0\\0&0\epm,\bpm \wt{B}_1\\0\epm,\bpm \wt{C}_1&0\epm,D\right),\]
where $(\wt{E}_{11},\wt{A}_{11})$ is regular.
\end{lem}
\begin{proof}
Our proof is based on the proof of {\cite[Theorem 1]{Safonov1990}}. Let $S_b$, $T_1$ and $\Gamma$ be as in \autoref{lem:sing}. By \autoref{lem:sing} we have $d:=$dim$(\ker(C_2))>0$. Therefore there exists an orthogonal $W\in\mb{R}^{r\times r}$ such that $(\widehat{C}_2~0_{p\times d})=C_2W^{\top}$ and $\ker(\widehat{C}_2)=\left\{0\right\}$ (e.g.\  singular value decomposition). We define $(\widehat{B}_2^{\top}~\widehat{B}_3^{\top}):= B_2^{\top}W^{\top}$ with appropriate dimensions. By \autoref{lem:sing} we have $\widehat{B}_3^{\top}=0$. Now we set $T_2:=\bsm I&0\\0&W\esm$, conclude $-A_{12}^{\top}\Gamma-C_2^{\top}C_1\Sigma_1=\sigma_1B_2B_1^{\top}$ from \eqref{eq:asg} and compute
\begin{EQ}[rll]
 T_2\opt{B}{S_b}{\sigma_1}&=T_2\Cont{S_b}B_b=\bpm I&0\\0&W\epm\bpm \Sigma_2 B_1\\-\sigma_1B_2\epm=\bpm \Sigma_2 B_1\\-\sigma_1\widehat{B}_2\\0_{d\times m}\epm&=\colon\bpm \wt{B}_1\\0_{d\times m}\epm\\
 \opt{C}{S_b}{\sigma_1}T_2^{-1}&=C_b\Obs{S_b}T_2^{-1}=\bpm C_1\Sigma_1&-\sigma_1C_2W^{\top}\epm\\
&=\bpm C_1\Sigma_1&-\sigma_1\widehat{C}_2&0_{p\times d}\epm&=\colon\bpm\wt{C}_1&0_{p\times d}\epm\\
 T_2\opt{E}{S_b}{\sigma_1}T_2^{-1}&=\bpm I&0\\0&W\epm\bpm \Gamma&0\\0&0\epm\bpm I&0\\0&W^{\top}\epm=\bpm \Gamma&0&0\\0&0&0\\0&0&0_{d\times d}\epm&=\colon\bpm \wt{E}_{11}&0\\0&0_{d\times d}\epm\\
 T_2\opt{A}{S_b}{\sigma_1}T_2^{-1}&=T_2\bpm -A_{11}^{\top}\Gamma-C_1^{\top}C_1\Sigma_1&\sigma_1C_1^{\top}C_2\\\sigma_1B_2B_1^{\top}&\sigma_1C_2^{\top}C_2\epm T_2^{-1}\\
&=\bpm -A_{11}^{\top}\Gamma-C_1^{\top}C_1\Sigma_1&\sigma_1C_1^{\top}\widehat{C}_2&0\\\sigma_1\widehat{B}_2B_1^{\top}&\sigma_1\widehat{C}_2^{\top}\widehat{C}_2&0\\0&0&0_{d\times d}\epm&=\colon\bpm \wt{A}_{11}&0\\0&0_{d\times d}\epm.
\end{EQ}
Since $\widehat{C}_2^{\top}\widehat{C}_2$ is regular, there exists $s\in\mb{C}$ such that the Schur complement
\[\left(sI+\left(A_{11}^{\top}+C_1^{\top}C_1\Sigma_1\Gamma^{-1}-\sigma_1C_1^{\top}\widehat{C}_2(\widehat{C}_2^{\top}\widehat{C}_2)^{-1}\widehat{B}_2B_1^{\top}\Gamma^{-1}\right)\right)\Gamma\]
of $s\wt{E}_{11}-\wt{A}_{11}$ in $-\sigma_1\widehat{C}_2^{\top}\widehat{C}_2$ is regular. \autoref{thm:schur} states that $s\wt{E}_{11}-\wt{A}_{11}$ is regular. Thus $(\wt{E}_{11},\wt{A}_{11})$ is regular. With $T:=T_2T_1^{-\top}$ the proof is finished.
\end{proof}

\begin{lem}\label{lem:Kkg2}
We set $S_1:=(\wt{E}_{11},\wt{A}_{11},\wt{B}_1,\wt{C}_1,D)$. For every $\omega\in\mb{R}$ with $\ie\omega\in\rho(\wt{E}_{11},\wt{A}_{11})$ the unequality
$\left\|\Fkt{S}(\ie\omega)-\Fkt{S_1}(\ie\omega)\right\|_2\leq \sigma_1$ 
holds. Moreover, $\Fkt{S_1}$ has no poles in $\mb{C}_{\geq0}$.
\end{lem}
\begin{proof}
First we show that  $\Fkt{\opt{\mc{S}}{S}{\gamma}}(s)\xrightarrow{\gamma\to\sigma_1+}\Fkt{S_1}(s)$ for all $s\in\mb{C}_{\geq0}$ except for a finite number. Let $T$ be as in \autoref{lem:singreg}. We rewrite $\Fkt{\opt{\mc{S}}{S}{\gamma}}(s)$ for $\gamma>\sigma_1$ and $s\in\mb{C}_{\geq0}$ as
\begin{EQ}[rl]
 \Fkt{\opt{\mc{S}}{S}{\gamma}}(s)&=\Fkt{\rse{TE^{-\top}}{\opt{\mc{S}}{S}{\gamma}}{T^{-1}}}(s)=\bpm \wt{C}_1&0\epm\left(TE^{-\top}(s\opt{E}{S}{\gamma}-\opt{A}{S}{\gamma})T^{-1}\right)^{-1}\bpm\wt{B}_1\\0\epm+D\\
&\stackrel{(*)}{=}\bpm \wt{C}_1&0\epm\left(\bpm s\wt{E}_{11}-\wt{A}_{11}&0\\0&0\epm+(\sigma_1^2-\gamma^2)(sI+TE^{-\top}A^{\top}T^{-1})\right)^{-1}\bpm\wt{B}_1\\0\epm+D,
\end{EQ}
where $(*)$ holds with
\begin{EQ}[rl]
\opt{E}{S}{\gamma}&=\opt{E}{S}{\sigma_1}+\opt{E}{S}{\gamma}-\opt{E}{S}{\sigma_1}=\opt{E}{S}{\sigma_1}+(\sigma_1^2-\gamma^2)E^{\top},\\
\opt{A}{S}{\gamma}&=\opt{A}{S}{\sigma_1}-(\sigma_1^2-\gamma^2)A^{\top}.
\end{EQ}
We set $\bsm M_1&M_2\\M_3&M_4\esm:=TE^{-\top}A^{\top}T^{-1}$ with appropriate dimensioned $M_i$. In addition, let $s\notin\sigma(M_4)\cup\sigma(\wt{E}_{11},\wt{A}_{11})$. Then
\begin{EQ}
 L_\gamma:=s\wt{E}_{11}-\wt{A}_{11}+(\sigma_1^2-\gamma^2)(sI+M_1)-(\sigma_1^2-\gamma^2)M_2(sI-M_4)^{-1}M_3
\end{EQ}
is the Schur complement of
\begin{EQ}[rl]
 TE^{-\top}(s\opt{E}{S}{\gamma}-\opt{A}{S}{\gamma})T^{-1}=\bpm s\wt{E}_{11}-\wt{A}_{11}+(\sigma_1^2-\gamma^2)(sI+M_1)&(\sigma_1^2-\gamma^2)M_2\\(\sigma_1^2-\gamma^2)M_3&(\sigma_1^2-\gamma^2)(sI+M_4)\epm
\end{EQ}
in $(\sigma_1^2-\gamma^2)(sI+M_4)$. By \autoref{thm:schur} the matrix $L_\gamma$ is regular because $(s\opt{E}{S}{\gamma}-\opt{A}{S}{\gamma})$ is regular. The functions  $\gamma\mapsto L_\gamma$ and thus $\gamma\mapsto \Fkt{\opt{\mc{S}}{S}{\gamma}}(s)=\wt{C}_1L_\gamma^{-1}\wt{B}_1+D$ are continuous in $(\sigma_1,\infty)$. Note that $L_{\sigma_1}=s\wt{E}_{11}-\wt{A}_{11}$ is regular. We obtain
\[\Fkt{\opt{\mc{S}}{S}{\gamma}}(s)=\wt{C}_1L_\gamma^{-1}\wt{B}_1+D\xrightarrow{\gamma\to\sigma_1+} \wt{C}_1(s\wt{E}_{11}-\wt{A}_{11})^{-1}\wt{B}_1+D=\Fkt{S_1}(s).\]
Therefore for all $\ie\omega\notin\sigma(M_4)\cup\sigma(\wt{E}_{11},\wt{A}_{11})$
\[\left\|\Fkt{S}(\ie\omega)-\Fkt{S_1}(\ie\omega)\right\|_2=\lim\limits_{\gamma\to\sigma_1+}\left\|\Fkt{S}(\ie\omega)-\Fkt{S_{S,\gamma}}(\ie\omega)\right\|_2\leq\lim\limits_{\gamma\to\sigma_1+}\gamma=\sigma_1.\]

Now we show that $\Fkt{S_1}$ has no poles in $\mb{C}_{\geq0}$. By \cite[Section 2.3]{Francis1987} we have $\left\|G(s)\right\|_2\leq\left\|G\right\|_\infty$ for all $G\in RH_\infty$ and $s\in\mb{C}_{\geq0}$. Thus we get the upper estimate
\begin{EQ}\label{eq:est}
 \left\|\Fkt{\opt{\mc{S}}{S}{\gamma}}(s)\right\|_2\leq \left\|\Fkt{\opt{\mc{S}}{S}{\gamma}}(s)-\Fkt{S}(s)\right\|_2+\left\|\Fkt{S}(s)\right\|_2 \leq \gamma+\left\|\Fkt{S}\right\|_\infty.
\end{EQ}
for all $s\in\mb{C}_{\geq0}$. 
Hence for all $s\in\mb{C}_{\geq0}\backslash(\sigma(\wt{E}_{11},\wt{A}_{11})\cup\sigma(M_4))$ we have
 \[\left\|\Fkt{S_1}(s)\right\|_2=\lim\limits_{\gamma\to\sigma_1+}\left\|\Fkt{\opt{\mc{S}}{S}{\gamma}}(s)\right\|_2\leq \lim\limits_{\gamma\to\sigma_1+}(\gamma+\left\|\Fkt{S}\right\|_\infty)=\sigma_1+\left\|\Fkt{S}\right\|_\infty\,.\]
Thus $\Fkt{S_1}$ has no poles in $\mb{C}_{\geq0}$. 
\end{proof}


\begin{thm}\label{thm:sing1}
The system $S_1$ given in \autoref{lem:Kkg2} solves \hyperref[p_p2]{(AP$_\infty$)}.
\end{thm}
\begin{proof}
We use the notations of the proof of \autoref{lem:singreg}. By \autoref{thm:causal2} and \autoref{lem:Kkg2} it remains to prove that $\sigma(\wt{E}_{11},\wt{A}_{11})\subseteq\mb{C}_{<0}\cup\{\infty\}$. Let $s_0\in\mb{C}_{\geq0}$. We apply \autoref{lem:pol2} because $s_0$ is not a pole  of $\Fkt{S_1}$. Let $x=\bpm x_1\\x_2\epm$ satisfy $0=\wt{C}_1x$ and
\begin{EQ}[rl]\label{eq:c}
0&=(s_0\wt{E}_{11}-\wt{A}_{11})x=\bpm s_0\Gamma+A_{11}^{\top}\Gamma+C_1^{\top}C_1\Sigma_1&-\sigma_1C_1^{\top}\widehat{C}_2\\-\sigma_1\widehat{B}_2B_1^{\top}&-\sigma_1\widehat{C}_2^{\top}\widehat{C}_2\epm\bpm x_1\\x_2\epm.
\end{EQ}
Hence
\begin{EQA}[rl]
0&=\bpm \wt{C}_1&0\epm\bpm x\\0\epm=\opt{C}{S_b}{\sigma_1}T_2^{-1}\bpm x\\0\epm, \label{eq:c2}\\ 
0&=\bpm s_0\wt{E}_{11}-\wt{A}_{11}&0\\0&0\epm\bpm x\\0\epm=T_2(s_0\opt{E}{S_b}{\sigma_1}-\opt{A}{S_b}{\sigma_1})T_2^{-1}\bpm x\\0\epm\\
&=T_2(s_0\opt{E}{S_b}{\sigma_1}+A_b^{\top}\opt{E}{S_b}{\sigma_1}+C_b^{\top}\opt{C}{S_b}{\sigma_1})T_2^{-1}\bpm x\\0\epm\overset{\eqref{eq:c2}}{=}T_2(s_0I+A_b^{\top})\opt{E}{S_b}{\sigma_1}T_2^{-1}\bpm x\\0\epm.
\end{EQA}
Due to $S_b\sim S\in\sys{n,p,m}^-$ we have $s_0\notin\sigma(-A_b^{\top})$ and the equation
\[0=\opt{E}{S_b}{\sigma_1}T_2^{-1}\bpm x\\0\epm=T_2\opt{E}{S_b}{\sigma_1}T_2^{-1}\bpm x\\0\epm=\bpm \Gamma&0&0\\0&0&0\\0&0&0\epm\bpm x_1\\x_2\\0\epm\]
follows. Thus we get $x_1=0$ by the regularity of $\Gamma$. Now we deduce 
\[\eqref{eq:c}\Rightarrow \sigma_1\widehat{C}_2^{\top}\widehat{C}_2x_2=0\Rightarrow x_2=0\]
from the regularity of $\widehat{C}_2^{\top}\widehat{C}_2$. Thus $x=0$. Analogously every $x\in\mb{R}^{n-d}$ which fulfils $0=(s_0\wt{E}_{11}-\wt{A}_{11})^{\top}x$ and $0=\wt{B}_1^{\top}x$ equals zero. Thus $s_0\notin\sigma(\wt{E}_{11},\wt{A}_{11})$ and $\sigma(\wt{E}_{11},\wt{A}_{11})\subseteq\mb{C}_{<0}\cup\{\infty\}$.
\end{proof}

The optimal solution given in the theorem above is still based on the balanced realization. Our second main contribution is the following result, which only needs one singular value decomposition.

\begin{thm}\label{thm:sing3}
Let $S=(E,A,B,C,D)\in\sys{n,p,m}^-$. Suppose $\penc{\opt{E}{S}{\sigma_1},\opt{A}{S}{\sigma_1}}$ given in \autoref{thm:causal3} is singular. Then there exist orthogonal $U,V\in\mb{R}^{n\times n}$ with
\begin{EQ}\label{eq:sing22}
 U\opt{A}{S}{\sigma_1}V=\bpm \widehat{A}_{11}&\widehat{A}_{12}\\0&0\epm\text{and regular }\widehat{A}_{11}\,.
\end{EQ}
Every regular $U,V\in\mb{R}^{n\times n}$ satisfying \eqref{eq:sing22}, fulfil
\begin{align*}
 U\opt{E}{S}{\sigma_1}V=\bpm \widehat{E}_{11}&\widehat{E}_{12}\\0&0\epm,&&U\opt{B}{S}{\sigma_1}=\bpm \widehat{B}_{1}\\0\epm,&&\opt{C}{S}{\sigma_1}V=\bpm \widehat{C}_{1}&\widehat{C}_{2}\epm.
\end{align*}

The system $S_2:=\left(\widehat{E}_{11},\widehat{A}_{11},\widehat{B}_1,\widehat{C}_1,D\right)$ solves \hyperref[p_p2]{(AP$_\infty$)}.
\end{thm}
\begin{proof}
Let $T\in\mb{R}^{n\times n}$ be as in \autoref{lem:singreg}. Then we obtain
\[\penc{TE^{-\top}\opt{E}{S}{\sigma_1}T^{-1},TE^{-\top}\opt{A}{S}{\sigma_1}T^{-1}}=\penc{\bpm \wt{E}_{11}&0\\0&0\epm,\bpm \wt{A}_{11}&0\\0&0\epm}\]
with $\wt{A}_{11}\in\mb{R}^{(n-d)\times (n-d)}$ and $\sigma(\wt{E}_{11},\wt{A}_{11})\subseteq\mb{C}_{<0}\cup\{\infty\}$. In particular, $\wt{A}_{11}$ is regular. The matrices $\opt{A}{S}{\sigma_1}$ and $TE^{-\top}\opt{A}{S}{\sigma_1}T^{-1}$ have the same rank. By singular value decomposition there exist regular orthogonal $U,V\in\mb{R}^{n\times n}$ such that
\begin{EQ}
 U\opt{A}{S}{\sigma_1}V=\bpm \widehat{A}_{11}&\widehat{A}_{12}\\0&0\epm \text{with regular } \widehat{A}_{11}\in\mb{R}^{(n-d)\times (n-d)}\,.
\end{EQ}
Let $U,V\in\mb{R}^{n\times n}$ be regular and satisfy \eqref{eq:sing22}. We define
\begin{align*}
 U\opt{E}{S}{\sigma_1}V=\colon\bpm \widehat{E}_{11}&\widehat{E}_{12}\\\widehat{E}_{21}&\widehat{E}_{22}\epm,&&U\opt{B}{S}{\sigma_1}=\colon\bpm \widehat{B}_{1}\\\widehat{B}_{2}\epm,&&\opt{C}{S}{\sigma_1}V=\colon\bpm \widehat{C}_1&\widehat{C}_2\epm.
\end{align*}

For all $s\in\mb{C}_{\geq0}$ the matrices $(s\widehat{E}_{11}-\widehat{A}_{11})$ and $(s\wt{E}_{11}-\wt{A}_{11})$ are regular. Thus we have
\begin{EQA}[rl]
\bpm s\widehat{E}_{11}-\widehat{A}_{11}&s\widehat{E}_{12}-\widehat{A}_{12}\\s\widehat{E}_{21}&s\widehat{E}_{22}\epm\left[\mb{R}^n\right]&=\left(UE^\top T^{-1}\bpm s\wt{E}_{11}-\wt{A}_{11}&0\\0&0\epm TV\right)\left[\mb{R}^n\right]\\
&\overset{(*)}{=}\left(UE^\top T^{-1}\right)\left[\mb{R}^{n-d}\times\left\{0\right\}^{d}\right]=\mb{R}^{n-d}\times\left\{0\right\}^{d}.~~~~~\label{eq:pro}
\end{EQA}
We obtain the last equality because $(*)$ holds for all $s\in\mb{C}_{\geq0}$ and in particular for $s=0$. That implies $\widehat{E}_{21}=0$, $\widehat{E}_{22}=0$ and finally $\sigma(\widehat{E}_{11},\widehat{A}_{11})=\sigma(\wt{E}_{11},\wt{A}_{11})$ which means $S_2\in\sys{n-d,p,m}^+$. Furthermore, we have $\widehat{B}_{2}=0$ because
\[\bpm \widehat{B}_{1}\\\widehat{B}_{2}\epm\left[\mb{R}^m\right]=\left(UE^{\top}T^{-1}\bpm \wt{B}_1\\0\epm\right)\left[\mb{R}^m\right]\overset{\eqref{eq:pro}}{=}\mb{R}^{n-d}\times\left\{0\right\}^{d}\,.\]

Finally, we show that $\Fkt{S_2}=\Fkt{S_1}$. Let $s\in\rho(\widehat{E}_{11},\widehat{A}_{11})$ and $Y_1:=\wt{C}_1(s\wt{E}_{11}-\wt{A}_{11})^{-1}$. We conclude with $Y:=\bpm Y_1&0\epm\in\mb{R}^{p\times n}$ and $X:=\bpm X_1&X_2\epm:=YTE^{-\top}U^{-1}$ that
\begin{EQ}[rl]
\bpm \widehat{C}_1&\widehat{C}_2\epm&=\bpm \wt{C}_1&0\epm TV=Y\bpm s\wt{E}_{11}-\wt{A}_{11}&0\\0&0\epm TV=X\bpm s\widehat{E}_{11}-\widehat{A}_{11}&s\widehat{E}_{12}-\widehat{A}_{12}\\0&0\epm.
\end{EQ}
In particular $X_1=\widehat{C}_1(s\widehat{E}_{11}-\widehat{A}_{11})^{-1}$. That implies
\begin{EQ}[rl]
 \Fkt{S_1}(s)&=\wt{C}_1(s\wt{E}_{11}-\wt{A}_{11})^{-1}\wt{B}_1=Y_1\wt{B}_1=\bpm Y_1&Y_2\epm\bpm\wt{B}_1\\0\epm=\bpm Y_1&Y_2\epm TE^{-\top}U^{-1}\bpm\widehat{B}_1\\0\epm\\
&=\bpm X_1&X_2\epm\bpm\widehat{B}_1\\0\epm=X_1\widehat{B}_1=\widehat{C}_1(s\widehat{E}_{11}-\widehat{A}_{11})^{-1}\widehat{B}_1=\Fkt{S_2}(s)\,.
\end{EQ}
By \autoref{thm:sing1} we get $\left\|\Fkt{S}-\Fkt{S_2}\right\|_\infty=\sigma_1.$
\end{proof}

Under the additional condition $E=I$ we can use the less complex Schur decomposition to attain \eqref{eq:sing22} in \autoref{thm:sing3}.
\begin{lem}
In the situation of \autoref{thm:sing3}, where $E=I$, we obtain regular matrices $U$ and $V:=U^{\top}$, such that \eqref{eq:sing22} holds, by applying the Schur decomposition to $\opt{A}{S}{\sigma_1}$. 
\end{lem}
\begin{proof}
 Let $T\in\mb{R}^{n\times n}$ be as in \autoref{lem:singreg}. Then we obtain
\[\penc{T\opt{E}{S}{\sigma_1}T^{-1},T\opt{A}{S}{\sigma_1}T^{-1}}=\penc{\bpm \wt{E}_{11}&0\\0&0\epm,\bpm \wt{A}_{11}&0\\0&0\epm}\]
with $\wt{A}_{11}\in\mb{R}^{(n-d)\times (n-d)}$ and $\sigma(\wt{E}_{11},\wt{A}_{11})\subseteq\mb{C}_{<0}\cup\{\infty\}$. In particular, $\wt{A}_{11}$ is regular. The matrices $\opt{A}{S}{\sigma_1}$ and $T\opt{A}{S}{\sigma_1}T^{-1}$ have the same eigenvalues counting multiplicities. By the Schur decomposition there exists a regular (in particular an orthogonal) $U\in\mb{R}^{n\times n}$ such that
\begin{EQ}
 U\opt{A}{S}{\sigma_1}U^{-1}=\bpm \widehat{A}_{11}&\widehat{A}_{12}\\0&\widehat{A}_{22}\epm=UT^{-1}\bpm \wt{A}_{11}&0\\0&0\epm TU^{-1}\,,
\end{EQ}
where $\widehat{A}_{11}\in\mb{R}^{(n-d)\times (n-d)}$ is regular and $\sigma(\widehat{A}_{22})=\left\{0\right\}$. This implies $\rank(\widehat{A}_{11})=\rank(\wt{A}_{11})=\rank(U\opt{A}{S}{\sigma_1}U^{-1})$ and thus $\widehat{A}_{22}=0$.
\end{proof}
In summary, we obtain the following result for \hyperref[p_p1]{(AP$_\infty$)}.
\begin{thm}\label{thm:loomin}
Let $S\in\sys{n,p,m}^0$. Then there exist $S_+\in\sys{n_+,p,m}^+$ and $S_-=(E_-,A_-,B_-,C_-,0)\in\sys{n_-,p,m}^-$ such that $S\sim S_+\oplus S_-$. We set $\sigma_1:=\sqrt{\max\sigma(E_-^{\top}\Obs{S_-} E_-\Cont{S_-})}$, then
\[\infb{\|\Fkt{S}-\Fkt{\widehat{S}}\|_\infty}{\widehat{S}\in\bigcup_{\widehat{n}\in\mb{N}}\sys{\widehat{n},p,m}^+}=\sigma_1\,.\]
We define the matrices 
\begin{EQ}[rlrlr]
\opt{R}{S_-}{\sigma_1}&:=\Obs{S_-}E_-\Cont{S_-}E_-^{\top}-\sigma_1^2I,&~~\opt{E}{S_-}{\sigma_1}&:=E_-^{\top}\opt{R}{S}{\sigma_1},&~~\opt{B}{S_-}{\sigma_1}:=E_-^{\top}\Obs{S_-}B_-,\\\opt{A}{S_-}{\sigma_1}&:=-A_-^{\top}\opt{R}{S}{\sigma_1}-C_-^{\top}\opt{C}{S_-}{\sigma_1},&~~\opt{C}{S_-}{\sigma_1}&:=C_-\Cont{S_-}E_-^{\top}.&
\end{EQ}
If $(\opt{E}{S_-}{\sigma_1},\opt{A}{S_-}{\sigma_1})$ is regular, then $S_+\oplus(\opt{E}{S_-}{\sigma_1},\opt{A}{S_-}{\sigma_1},\opt{B}{S_-}{\sigma_1},\opt{C}{S_-}{\sigma_1},0)$ solves \hyperref[p_p1]{(AP$_\infty$)}.

If $(\opt{E}{S_-}{\sigma_1},\opt{A}{S_-}{\sigma_1})$ is singular, then there exist orthogonal matrices $U,V\in\mb{R}^{n\times n}$ with
\begin{EQ}\label{eq:sing212}
 U\opt{A}{S_-}{\sigma_1}V=\bpm \widehat{A}_{11}&\widehat{A}_{12}\\0&0\epm \text{and regular } \widehat{A}_{11}\,.
\end{EQ}
Every regular $U,V\in\mb{R}^{n\times n}$ satisfying \eqref{eq:sing212}, fulfil
\begin{align*}
 U\opt{E}{S_-}{\sigma_1}V=\bpm \widehat{E}_{11}&\widehat{E}_{12}\\0&0\epm,&&U\opt{B}{S_-}{\sigma_1}=\bpm \widehat{B}_{1}\\0\epm,&&\opt{C}{S_-}{\sigma_1}V=\bpm \widehat{C}_{1}&\widehat{C}_{2}\epm.
\end{align*}
The system $S_+\oplus(\widehat{E}_{11},\widehat{A}_{11},\widehat{B}_{1},\widehat{C}_{1},0)$ solves \hyperref[p_p1]{(AP$_\infty$)}.
\end{thm}


\section{Application}\label{sec:num}

\subsection*{Algorithm}

We rewrite our results in an algorithmic manner. Note that the solution of \hyperref[p_p1]{(AP$_2$)} is necessary to compute the solution of \hyperref[p_p1]{(AP$_\infty$)} and thus no extra computation is needed.

\textbf{Input:} $S\in\sys{n,p,m}^0$

\textbf{Output:} $S_q\in\sys{n_q,p,m}^+$ which solves \hyperref[p_p1]{(AP$_q$)} for $q\in\{2,\infty\}$. 

\begin{alg}
\begin{enumerate}
	\item Decompose $S$ into $S_-\in\sys{n_-,p,m}^-$ and $S_+\in\sys{n_+,p,m}^+$ as in \autoref{thm:stabinst} with the help of \autoref{thm:decsyl}.
	\item The system $S_2:=S_+$ solves \hyperref[p_p1]{(AP$_2$)}.
	\item Determine $\wt{\mathfrak{C}}:=\Cont{S_-}E_-^{\top}$, $\wt{\mathfrak{O}}:=E_-^{\top}\Obs{S_-}$ and $\sigma_1=\sqrt{\max\sigma(\wt{\mathfrak{O}}^{\top}\wt{\mathfrak{C}})}$.
\item \(
\begin{aligned}[t]
 \text{Compute }R&:=\wt{\mathfrak{O}}^{\top}\wt{\mathfrak{C}}-\sigma_1^2I\,,&~&\widehat{E}:=E_-^{\top}R\,,&~&\widehat{B}:=\wt{\mathfrak{O}}B_-\,,\\
&&~&\widehat{A}:=-A_-^{\top}R-C_-^{\top}\widehat{C}\,,&~&\widehat{C}:=C_-\wt{\mathfrak{C}}\,.
\end{aligned}\)
	\item If $(\widehat{E},\widehat{A})$ is regular (test with \autoref{c_singreg}), then \hyperref[p_p1]{(AP$_\infty$)} is solved by $S_\infty:=S_+\oplus(\widehat{E},\widehat{A},\widehat{B},\widehat{C},0)$.
	\item If $(\widehat{E},\widehat{A})$ is singular, then determine $U,V\in\mb{R}^{n\times n}$ (SVD) such that
\begin{EQ}
 U\widehat{A}V=\bpm \widehat{A}_{11}&\widehat{A}_{12}\\0&0\epm\text{ with regular }\widehat{A}_{11}.
\end{EQ}
Compute $\widehat{E}_{11},\widehat{B}_{1},\widehat{C}_{1}$ of
$U\widehat{E}V=\bpm \widehat{E}_{11}&\widehat{E}_{12}\\0&0\epm,~U\widehat{B}=\bpm \widehat{B}_{1}\\0\epm,~\widehat{C}V=\bpm \widehat{C}_{1}&\widehat{C}_{2}\epm,$
then \hyperref[p_p1]{(AP$_\infty$)} is solved by $S_\infty:=S_+\oplus(\widehat{E}_{11},\widehat{A}_{11},\widehat{B}_{1},\widehat{C}_{1},0)$.
\end{enumerate}
\end{alg}

For the numerical issues of step 1 we refer to \cite{Kaagstroem1992} and of step 3 to \cite{Penzl1998}.

Both approximation techniques does not increase the order of the system, i.\,e. $n_q\leq n$. For optimal $RH_2$-approximation we even have $n_2<n$. Another advantage is the less computational complexity compared to $RH_\infty$-approximation. The behaviour of the transfer functions $\Fkt{S}$ and $\Fkt{S_2}$ at infinity is the same, i.\,e. $\Fkt{S}(\infty)=\Fkt{S_2}(\infty)$. This does not hold for optimal $RH_\infty$-approximation in general. However, the suboptimal approximation $S_+\oplus\opt{\mc{S}}{S_-}{\gamma}$ (choosing $\gamma>\sigma_1$ in \autoref{thm:causal2}) has this property. Since $\opt{E}{S_-}{\gamma}$ is regular, we have
\[\Fkt{S}(\infty)=\Fkt{S_+}(\infty)=\Fkt{S_+}(\infty)+\Fkt{\opt{\mc{S}}{S_-}{\gamma}}(\infty)=\Fkt{S_+\oplus\opt{\mc{S}}{S_-}{\gamma}}(\infty)\,.\]

\subsection*{Examples}

The following benchmarks are performed in \matlab. We test the building model (\autoref{fig:build}) and the clamped beam model (\autoref{fig:beam}), see \cite{Chahlaoui2002} for details. Both models are described by stable standard systems with one input and one output, i.e.\ $m=p=1$. We apply (shifted) Arnoldi method to reduce the systems. The parameters (interpolation point $s\in\mb{C}\cup\{\infty\}$ and reduced order $k\in\mb{N}$) are based on the examples in \cite{Antoulas2001}. The computed reduced systems are unstable as listed in \autoref{tab:build}. Thus we apply the algorithm presented above. Since the transfer functions of the original systems (and hence of the reduced systems) vanish at infinity an optimal $RH_\infty$-approximation would result in huge relative errors at high frequencies. That is why we compute suboptimal approximations (i.\,e.\ $\gamma>\sigma_1$) to match the original transfer functions at infinity.

\begin{table}[ht]
\centering
\begin{tabular}{|l|c|l|c|c|c|}
\hline model 	& order	& model reduction& reduced&  unstable 	&max. real part\\
		&  	& method &  order &  poles 	&of unstable poles\\
\hline 	building& 48  	& Arnoldi, $s=\infty$ &  31 & 2 &$\approx 42.4$ \\
		&	& $RH_2$  &  29 & -& - \\
		& 	& $RH_\infty$, $\gamma=1.001\cdot\sigma_1$ &  31 &- &- \\
\hline 	clamped & 348 	& Arnoldi, $s=0.1$ & 13 & 1& $\approx 1.5$ \\
       	beam	&	& $RH_2$  & 12 &- &- \\
		& 	& $RH_\infty$, $\gamma=1.001\cdot\sigma_1$  &  13 &- &-  \\
\hline
\end{tabular}
\caption{Summary of the results}
\label{tab:build}
\end{table}

The error between the original systems and the reduced systems of Arnoldi method does not significantly chance after applying our stabilization algorithm as depicted in \autoref{fig:build} and \autoref{fig:beam}. 


\begin{figure}[ht]
 	\centering
	\matpdf{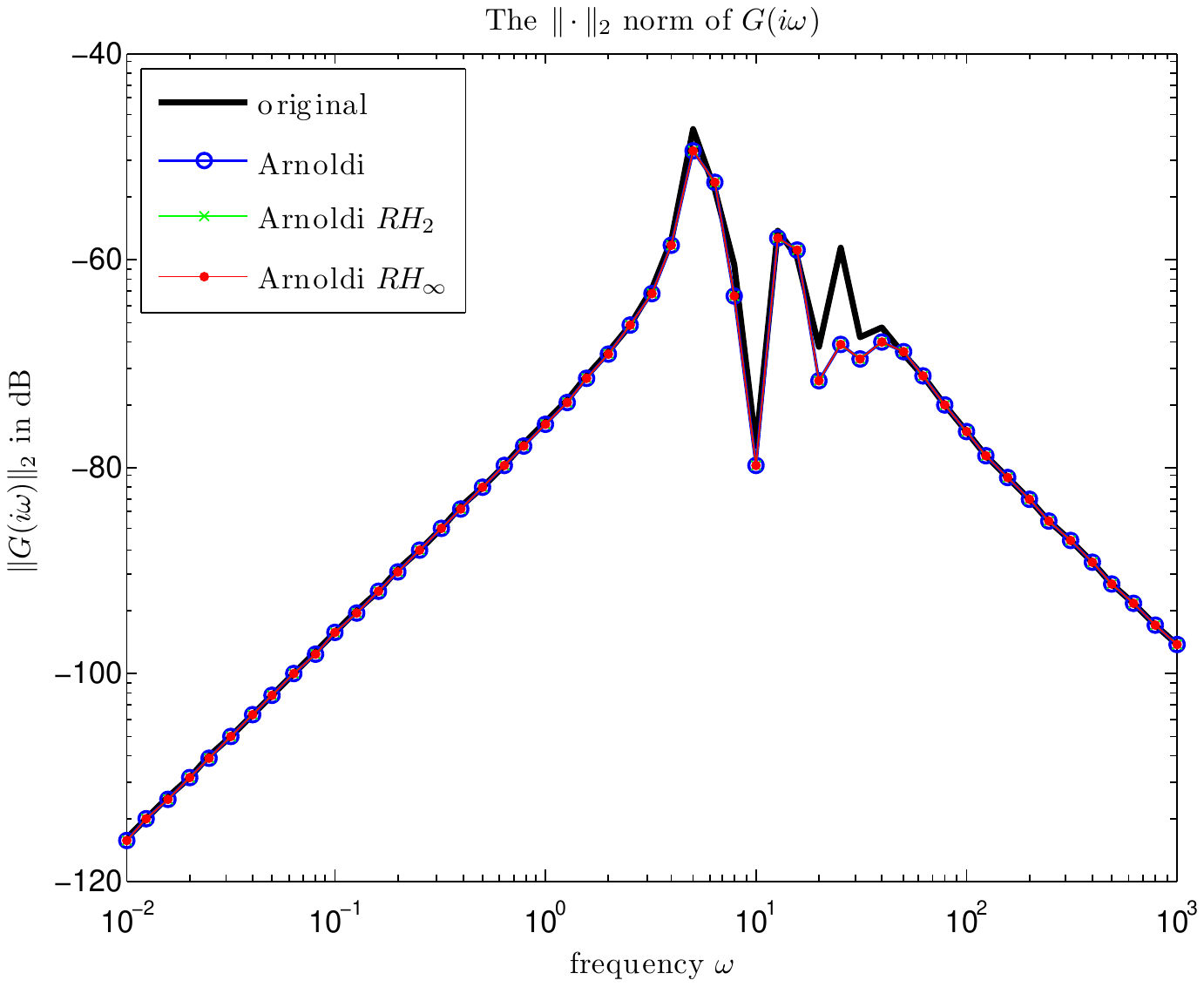}\quad\matpdf{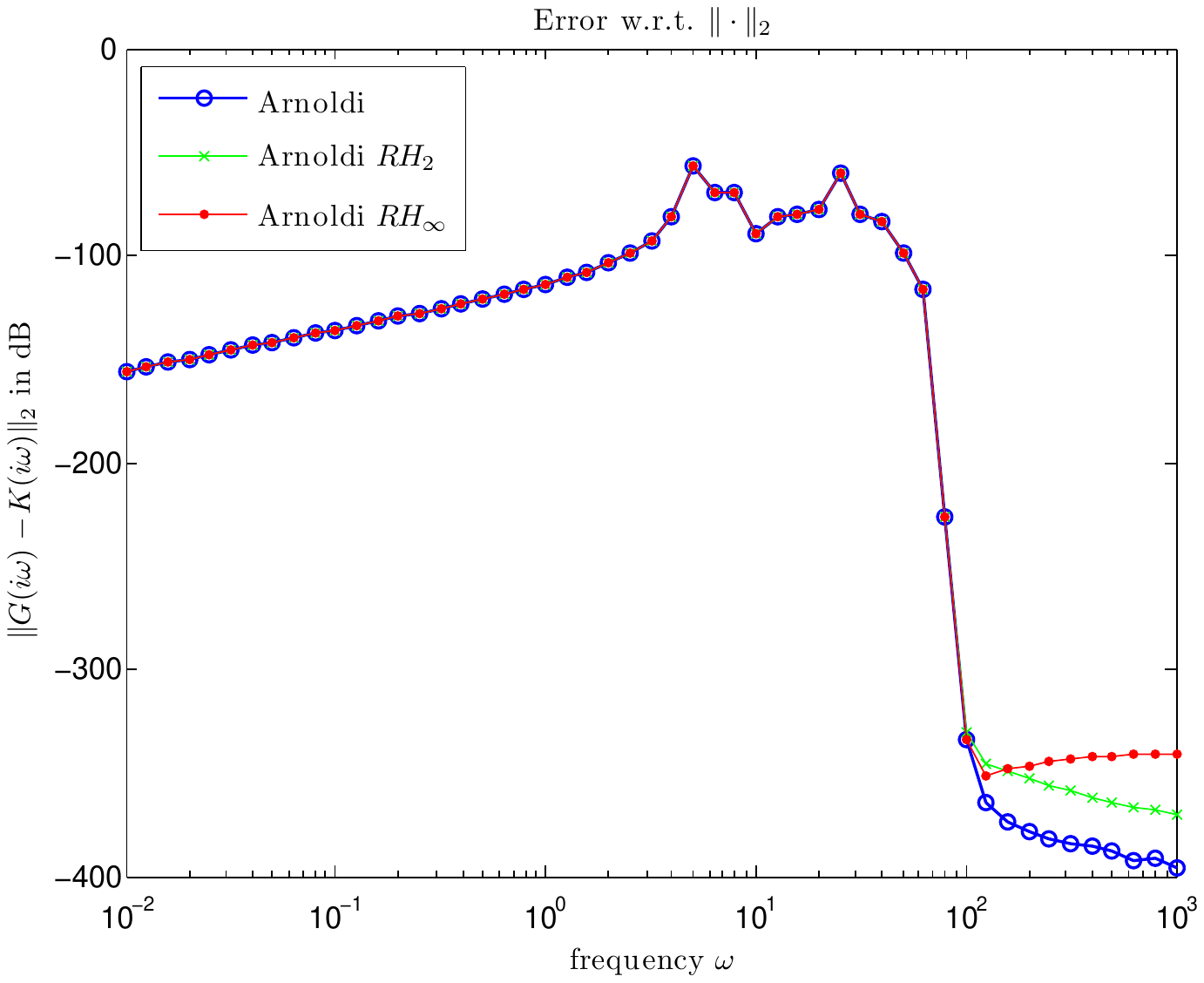}
	\caption{Frequency responses of the (reduced) building model (left) and the error systems (right)}
	\label{fig:build}
\end{figure}


\begin{figure}[ht]
 	\centering
	\matpdf{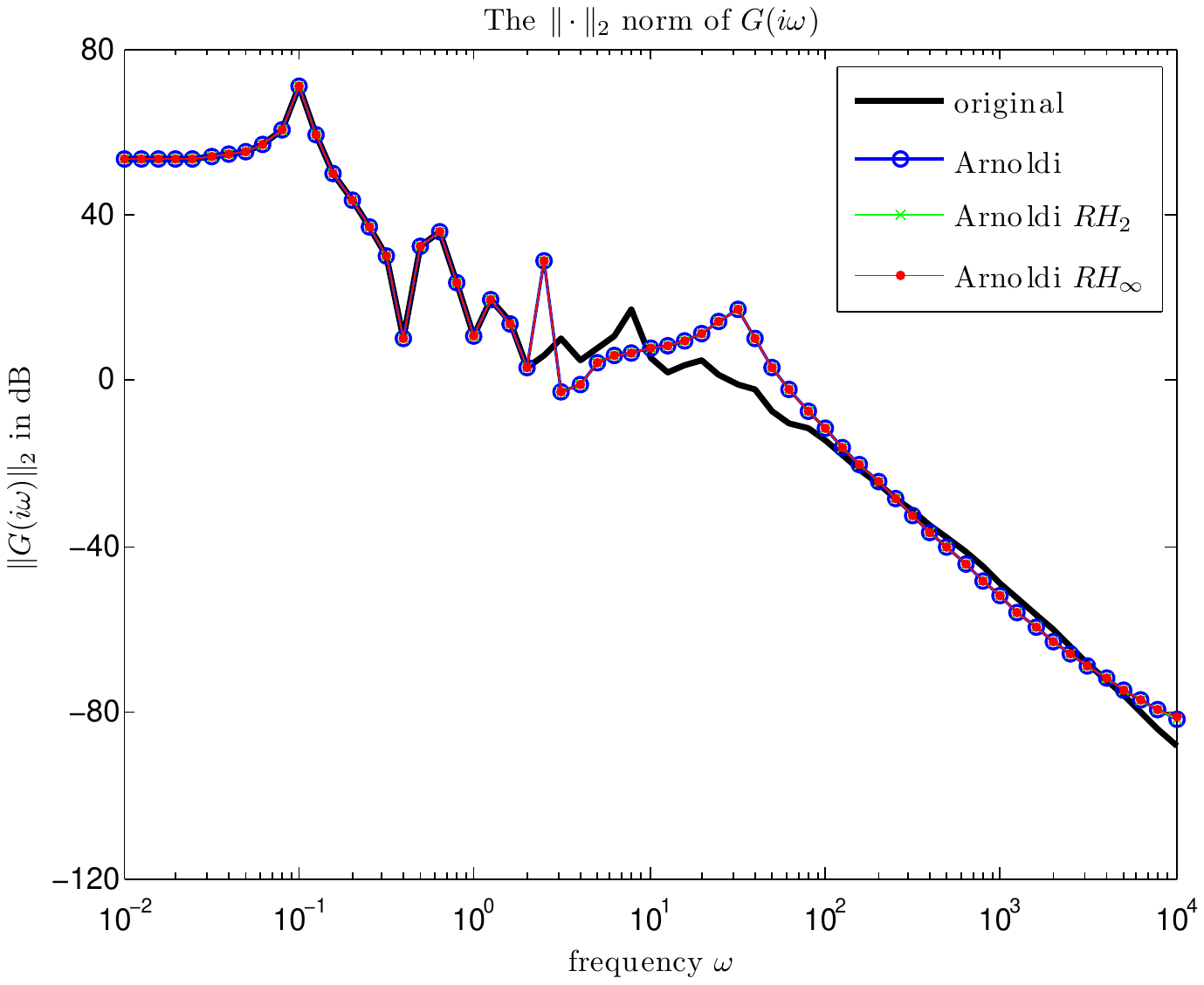}\quad\matpdf{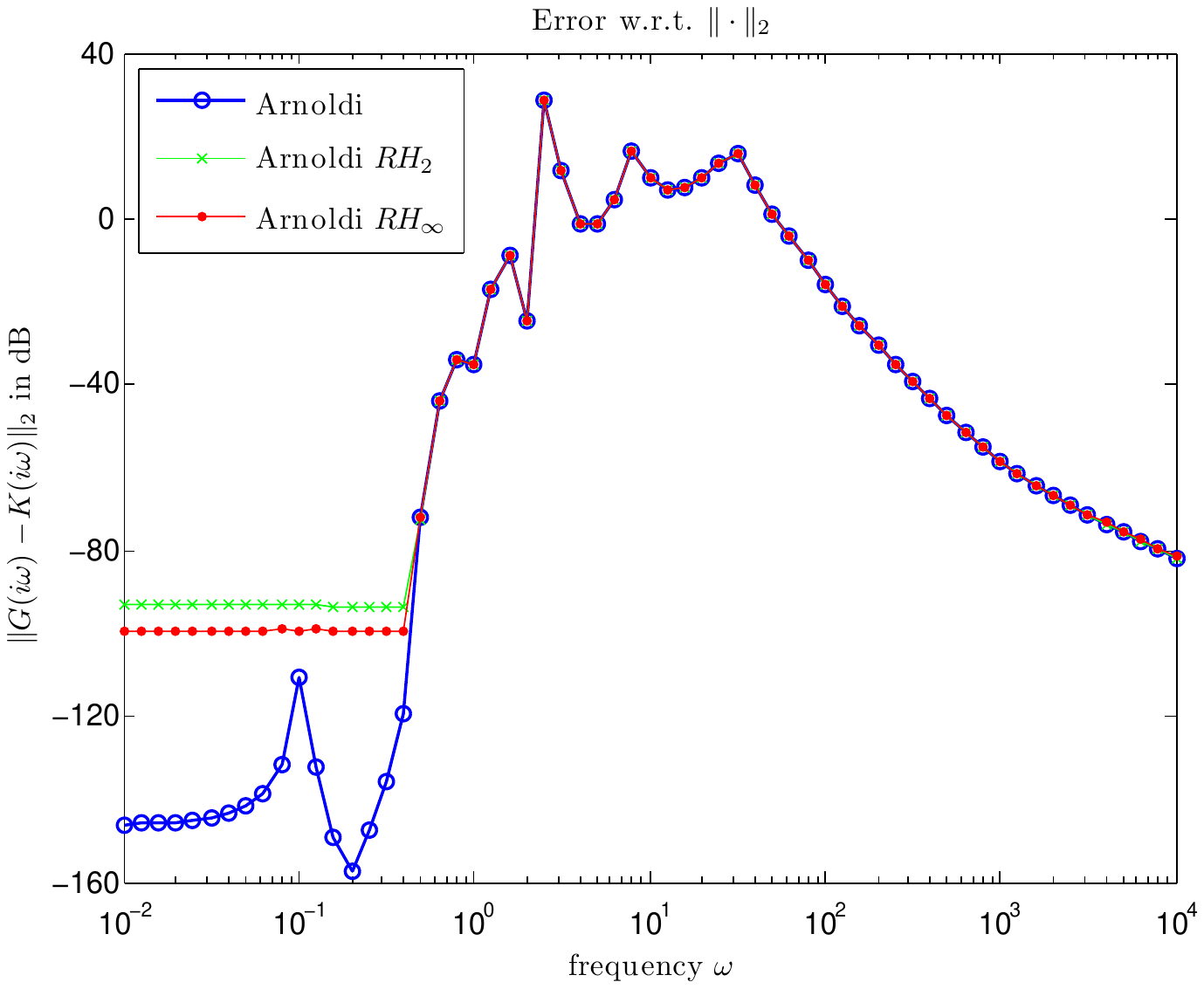}
	\caption{Frequency responses of the (reduced) clamped beam model (left) and the error systems (right)}
	\label{fig:beam}
\end{figure}




\phantomsection
	\addcontentsline{toc}{section}{Literature}
\bibliography{Literatur}

\end{document}